\documentclass[11pt,reqno]{amsart}

\usepackage{amsmath,amsthm,amssymb}
\usepackage{amssymb}

\usepackage{graphics,graphicx}
\usepackage{hyperref}
\usepackage[usenames, dvipsnames]{xcolor}

\definecolor{darkblue}{rgb}{0.0,0.0,0.3}
\hypersetup{colorlinks,breaklinks,
  linkcolor=darkblue,urlcolor=darkblue,
anchorcolor=darkblue,citecolor=darkblue}

\usepackage[square,sort,comma,numbers]{natbib}
\usepackage{fancyvrb}
\usepackage{enumitem}

\usepackage{booktabs}
\usepackage{threeparttable}

\theoremstyle{plain}
\newtheorem{theorem}{Theorem}[section]
\newtheorem*{theorem*}{Theorem}
\newtheorem{lemma}[theorem]{Lemma}
\newtheorem{proposition}[theorem]{Proposition}
\newtheorem*{proposition*}{Proposition}
\newtheorem{corollary}[theorem]{Corollary}
\newtheorem*{corollary*}{Corollary}

\theoremstyle{definition}
\newtheorem{remark}[theorem]{Remark}

\numberwithin{equation}{section}

\renewcommand{\Im}{\operatorname{Im}}
\renewcommand{\Re}{\operatorname{Re}}

\DeclareMathOperator*{\Res}{Res}
\DeclareMathOperator*{\Vol}{Vol}

\DeclareMathOperator{\PSL}{PSL}

\title[The Laplace Transform of the Square for the Circle Problem]{The Laplace
Transform of the Second Moment in the Gauss Circle Problem}

\author[Hulse]{Thomas A. Hulse}
  \address{Boston College}
\author[Kuan]{Chan Ieong Kuan}
  \address{Sun Yat-Sen University}
\author[Lowry-Duda]{David Lowry-Duda}
  \address{ICERM and Brown University}
  \email{davidlowryduda@brown.edu}
  \thanks{This material is based upon work supported by the National Science
    Foundation Graduate Research Fellowship Program under Grant No. DGE 0228243.
    David also gratefully acknowledges support from EPSRC Programme Grant
    EP/K034383/1 LMF:\ L-Functions and Modular Forms and support from
    the Simons Collaboration in Arithmetic Geometry, Number Theory, and
    Computation via the Simons Foundation grant 546235.}
\author[Walker]{Alexander Walker}
  \address{Rutgers University}
  \thanks{This material is based upon work supported by the National Science
  Foundation under Grant No. DMS-1440140 while two of the authors were in
  residence at the Mathematical Sciences Research Institute in Berkeley,
  California, during the Spring 2017 semester.}

\date{\today}

\begin{document}

\begin{abstract}

The Gauss circle problem concerns the difference $P_2(n)$ between the area of a
circle of radius $\sqrt{n}$ and the number of lattice points it contains.
In this paper, we study the Dirichlet series with coefficients $P_2(n)^2$, and
prove that this series has meromorphic continuation to $\mathbb{C}$.
Using this series, we prove that the Laplace transform of $P_2(n)^2$ satisfies
$\int_0^\infty P_2(t)^2 e^{-t/X} \, dt = C X^{3/2} -X + O(X^{1/2+\epsilon})$,
which gives a power-savings improvement to a previous result of
Ivi\'c~\cite{Ivic1996}.

Similarly, we study the meromorphic continuation of the Dirichlet series
associated to the correlations $r_2(n+h)r_2(n)$, where $h$ is fixed and
$r_2(n)$ denotes the number of representations of $n$ as a sum of two squares.
We use this Dirichlet series to prove asymptotics for $\sum_{n \geq 1}
r_2(n+h)r_2(n) e^{-n/X}$, and to provide an additional evaluation of the
leading coefficient in the asymptotic for $\sum_{n \leq X} r_2(n+h)r_2(n)$.
\end{abstract}

\maketitle

\section{Introduction}

A classic result of Gauss states that the number $S_2(R)$ of integer lattice points contained in a circle of radius $\sqrt{R}$ is well-approximated by the circle's area.
To quantify the accuracy of this estimate, one defines the lattice point discrepancy
\begin{equation*}
  P_2(R):= S_2(R)-\pi R = \sum_{n \leq R} r_2(n) - \pi R,
\end{equation*}
in which $r_2(n)$ denotes the number of representations of $n$ as a sum of two integer squares.

The famous \emph{Gauss circle problem} is the pursuit of the minimal $\alpha$ for which $P_2(R) \ll R^{\alpha+\epsilon}$ for all $\epsilon > 0$.
Pointwise, the greatest improvement to the trivial bound $P_2(R) \ll \sqrt{R}$ of Gauss is due to Huxley~\cite{Huxley03}, who proved
\begin{equation*}
  P_2(R) \ll R^{131/416} (\log R)^{18637/8320} \qquad (\tfrac{131}{416} = 0.31490\ldots)
\end{equation*}
using his variant of the ``discrete Hardy-Littlewood circle method.''

Lower bounds in the form of $\Omega_\pm$-results suggest the well-known
conjecture $P_2(R) \ll R^{1/4+\epsilon}$.  This conjecture is also supported by
various on-average results, including \emph{mean square estimates}, which are
estimates of the form
\begin{equation}\label{eq:P2_mean_square_sharp}
  \int_0^R P_2(t)^2 \, dt =
  \Big( \frac{1}{3 \pi^2} \sum_{n \geq 1} \frac{r_2^2(n)}{n^{3/2}} \Big) R^{\frac{3}{2}}
  + Q(R),
\end{equation}
where $Q(R)$ is an error term.
The current best bound for $Q(R)$ is due to Nowak~\cite{Nowak04}, who showed that
\begin{equation*}
  Q(R) \ll R (\log R)^{\frac{3}{2}} \log \log R.
\end{equation*}

In~\cite{Ivic1996}, Ivi\'{c} considered the Laplace transform of $P_2(R)^2$ (as well as the second moment of the error in the Dirichlet divisor problem) and proved
\begin{equation}\label{eq:laplace_intro}
  \int_0^\infty P_2(t)^2 e^{-t/R} \, dt =
  \frac{1}{4} \Big( \frac{R}{\pi} \Big)^{\frac{3}{2}} \sum_{n \geq 1} \frac{r_2^2(n)}{n^{3/2}} - R + O_\epsilon(R^{\alpha + \epsilon}),
\end{equation}
where $\alpha$ is chosen such that the convolution estimate
\begin{align}
  \sum_{n \leq R} r_2(n) r_2(n+h)
  &=
  C_h R + O(R^{\alpha+\epsilon}) \label{eq:convolution_ivic}
\end{align}
holds uniformly for $h \leq \sqrt{X}$.
In this way, improved asymptotics for the convolution sum~\eqref{eq:convolution_ivic} lead to sharper asymptotics for the Laplace transform of $P_2(n)^2$.
In~\cite{Ivic2001}, Ivi\'{c} built on these techniques and recent results of Chamizo~\cite{Chamizo1999} to adapt an argument of Motohashi concerning convolution sums in the divisor problem~\cite{Motohashi1994}, and showed that one can take $\alpha \leq \frac{2}{3}$ in~\eqref{eq:laplace_intro}.
Thus the current best error term in the Laplace transform for $P_2(R)^2$ in~\eqref{eq:laplace_intro} is $O(R^{2/3 + \epsilon})$.

The primary result in this article is the following theorem, which establishes an improved error term in the above mean square Laplace transform.

\begin{theorem}
  For any $\epsilon > 0$,
  \begin{equation*}
    \int_0^\infty P_2(t)^2 e^{-t/R} \, dt =
    \frac{1}{4} \Big( \frac{R}{\pi} \Big)^{\frac{3}{2}}
    \sum_{n \geq 1} \frac{r_2^2(n)}{n^{3/2}} - R
    + O_\epsilon(R^{\frac{1}{2} + \epsilon}).
  \end{equation*}
\end{theorem}

The exponent $\frac{1}{2}$ in the error term is optimal, and arises from a line
of spectral poles appearing in our analysis. Assuming the Riemann Hypothesis, it
is straightforward to remove the $\epsilon$ from the error term. This leads to
Theorem~\ref{thm:optimal_error}.

We approach this problem by investigating the Dirichlet series associated to $S_2(R)^2$ and $P_2(R)^2$, defined by
\begin{equation*}
  D(s, S_2 \times S_2) := \sum_{n \geq 1} \frac{S_2(n)^2}{n^{s + 2}}, \quad D(s, P_2 \times P_2) := \sum_{n \geq 1} \frac{P_2(n)^2}{n^s}.
\end{equation*}
These Dirichlet series have been partially analyzed before.
For example, a recent paper of Furuya and Tanigawa~\cite{Furuya2014} builds upon the earlier work of Ivi\'c to give a partial meromorphic continuation of the Dirichlet series $D(s,P_2 \times P_2)$.
In this paper, techniques developed in~\cite{HulseKuanLowryDudaWalker17}, \cite{HKLDWshortsums}, and~\cite{HulseGaussSphere} are applied to derive the full meromorphic continuation of $D(s,P_2 \times P_2)$. 
It is further possible to show that $D(s, P_2 \times P_2)$ is of polynomial growth in vertical strips, allowing straightforward analysis from integral transforms.

Let $B_k(\sqrt R)$ denote the $k$-dimensional ball of radius $\sqrt R$, let $r_k(n)$ denote the number of representations of $n$ as a sum of $k$ squares, and define
\begin{equation*}
  S_k(R) := \sum_{n \leq R} r_k(n), \quad P_k(R) := \sum_{n \leq R} r_k(n) - \Vol B_k(\sqrt R).
\end{equation*}
Estimating $P_k(R)$ represents the $k$-dimensional analogue of the Gauss circle problem, described in detail in the survey article~\cite{IvicSurvey2006}.
In~\cite{HulseGaussSphere}, the authors showed that for $k \geq 3$, the Dirichlet series $D(s, S_k \times S_k)$ and $D(s, P_k \times P_k)$ have meromorphic continuation to the complex plane.
These continuations were used to prove $k$-dimensional analogues of~\eqref{eq:P2_mean_square_sharp} and~\eqref{eq:laplace_intro} in the case $k \geq 3$.

Analysis of the function $\mathcal{V}(z) \approx \lvert \theta^2(z) \rvert^2$,
where $\theta(z)$ is the standard Jacobi theta function, forms the heart of this
paper.
Some of the techniques used in~\cite{HulseGaussSphere} to understand
$D(s, P_k \times P_k)$ for $k \geq 3$ (corresponding to studying
$\lvert \theta^k(z) \rvert^2$) apply directly in the case when
$k = 2$, and we try to indicate these parallels when applicable.
Some techniques can be vastly simplified, and others break down.

Many of these differences stem from the Eisenstein series used to
study the behavior of $\mathcal{V}(z)$ at the cusps of $\Gamma_0(4)$.
For dimensions $k \geq 3$, it is possible to use real analytic
Eisenstein series with carefully chosen spectral parameter, but in dimension
$k=2$ the analogous Eisenstein series diverges.
We instead work with the constant terms of the Laurent expansions of these real
analytic Eisenstein series, which retain the automorphic properties. From this we are able to show that this more complicated behavior of
$\mathcal{V}(z)$ at the cusps leads to the existence of double poles in
associated $L$-functions, which in turn perfectly cancel out in the
analysis leading to the proof of our main theorems.

We note that $\mathcal{V}(z)$ is morally the norm of $\theta^2(z)$, a
full-integral weight modular form on $\Gamma_0(4)$ with multiplicative
coefficients. This allows some otherwise technical arguments to be simplified or
reduced to a comparison of Euler products (e.g.\ the computation of
the residues in section \S\ref{sec:DsP2P2}), in contrast to what we would expect
from analogous arguments in dimension $k \geq 3$.
One major instance where the dimension $k = 2$ case is cleaner than for $k \geq
3$ is when dealing with the Rankin-Selberg convolution $\sum r_k(n)^2 n^{-s}$;
the appendix to~\cite{HulseGaussSphere} is a technical argument describing a
method to understand this convolution series, but for dimension $2$ this sum can
be readily described in terms of $\zeta(s)$ and $L(s, \chi)$.

In this work, we show how to modify and extend previous methods to address
the dimension $2$ case.
This culminates in Theorem~\ref{thm:DsP2xP2_properties}, which describes the
meromorphic continuation of $D(s, P_2 \times P_2)$ to the entire complex plane.

The techniques of this paper can also be used to give explicit meromorphic
continuation the shifted convolution Dirichlet series,
\begin{equation*}
  D_2(s; h) := \sum_{n \geq 0} \frac{r_2(n)r_2(n+h)}{(n+h)^s}.
\end{equation*}
These shifted convolutions give a new way towards understanding Chamizo's
asymptotic~\eqref{eq:convolution_ivic} and give a new derivation of the constant
$C_h$.
With the aid of exponential smoothing, particularly strong smoothed versions
of~\eqref{eq:convolution_ivic} are attainable.
\begin{theorem}
  For any $\epsilon > 0$,
  \begin{equation*}
    \sum_{n \geq 1} r_2(n) r_2(n+h) e^{-n/X} = C_h X + O_\epsilon\big( X^{\frac{1}{2} + \epsilon}e^{h/X} h^\Theta \big).
  \end{equation*}
\end{theorem}

Here, $\Theta$ refers to the best-known progress towards the non-archimedean
Ramanujan conjecture. A full statement of this result, including a non-trivial
estimate for the corresponding sharp sum~\eqref{eq:convolution_ivic} and a new
evaluation of $C_h$, is given in Theorem~\ref{thm:chamizo}. While Chamizo also
used spectral techniques, including trace-type formulas, to evaluate the leading
coefficient, our methods are very different.

Although $D_2(s;h)$ can be used to provide bounds for the shifted convolution
sum~\eqref{eq:convolution_ivic}, the authors have not been able to improve upon
Chamizo's bound for that sum.

However, by summing over both $n$ and $h$, we gain deep understanding of
\begin{equation*}
  Z_2(s,w) := \sum_{h \geq 1} \sum_{n \geq 0} \frac{r_2(n) r_2(n+h)}{(n+h)^s h^w},
\end{equation*}
which can be used to recognize significant cancellation within $D(s, P_k \times
P_k)$. The reason our analysis for $D(s, P_k \times P_k)$ allows improvement of
known bounds for the Laplace transform of the second
moment~\eqref{eq:laplace_intro} without also yielding improvements over
Chamizo's bound for the shifted convolution~\eqref{eq:convolution_ivic} for each
fixed $h$ is due to the averaging over both $n$ and $h$; the obstructions
coming from $D_s(s;h)$ explicitly combine and cancel. See
Remark~\ref{remark:chamizo_compare} for more on this aspect.

\section*{Directions for Further Research}

It is natural to ask whether these techniques could be applied
to Dirichlet's divisor problem and its variants, perhaps by replacing the
$r_2(n)$ coefficients in the above constructions with the sum-of-divisors
functions $\sigma_{k-1}(n) = \sum_{d \mid n} d^{k-1}$.
These appear as the coefficients of both holomorphic and special values of
real-analytic Eisenstein series.
Both methods have their challenges: the holomorphic object becomes more
complicated when $k=1$, and the real analytic object has a Fourier expansion
comprised of $K$-Bessel functions instead of the much simpler $e^{2\pi i z}$.

We might also investigate the problem of counting integer lattice points inside
of certain ellipses.
If we replace $\theta^2(z)$ with $\theta(a^2z)\theta(b^2z)$ for $a,b\in
\mathbb{N}$, we would be able to count the number of integer lattice points
inside a growing ellipse with major and minor radii of integer proportion.
Additionally, by considering the Fourier coefficients of theta functions of
binary quadratic forms we might be able to count the the number of lattice
points inside level curves corresponding to those quadratic forms.
It would also be interesting to consider general Jacobi theta functions.

More broadly, multiple Dirichlet series continue to have wide-ranging
applications in number theory.
In particular, it has recently been discovered by the authors that for a fixed,
square-free $t \in \mathbb{N}$, the asymptotics of the partial sum
\begin{equation*}
  \sum_{m,n \in X} r_1(m-n)r_1(m+n)r_1(m)r_1(tm),
\end{equation*}
where $r_1(m)$ is the $m$-th Fourier coefficient of $\theta(z)$, are deeply
related to whether $t$ is a congruent number --- there is a main term if and
only if $t$ is congruent, and the error term is connected to the rank of an
associated elliptic curve.
By decomposing a corresponding multiple Dirichlet series into manageable
variants of $D_1(s;h)$, the authors hope to build on the methods of this paper
in order to gain new insight into the Congruent Number Problem.
Preliminary results can be found in~\cite{ChanEtAlCongruent}.

\section*{Acknowledgments}

We would like to thank Aleksandar Ivi\'{c} for suggesting that we consider this problem, and for his helpful remarks.
We also thank the referee for many valuable comments, and for pointing out the classical estimates leading to Remark~\ref{remark:ref} and simplifications in the proof of Lemma~\ref{lem:laplace_decomp_3}.

\section{Decomposition of $D(s, S_2 \times S_2)$ and $D(s, P_2 \times P_2)$}

In this section, we show that the meromorphic properties of $D(s, P_2 \times P_2)$ can be recovered from the meromorphic properties of $D(s, S_2 \times S_2)$.
We then decompose $D(s, S_2 \times S_2)$ into a sum of simpler functions that we analyze in later sections.
The methodology of this section is extremely similar to section \S2 of~\cite{HulseGaussSphere}, so we sketch the proofs and focus on the differences.

\begin{proposition}\label{prop:DsPP_to_DsSS}
The Dirichlet series $D(s,P_2 \times P_2)$ and $D(s,S_2 \times S_2)$ are related through the equality
  \begin{align*}
  D(s, P_2 \times P_2) &= D(s-2, S_2 \times S_2) +\pi^2 \zeta(s-2)\\
  &\quad -2 \pi \zeta(s-1) -2\pi L(s-1,\theta^2) \\
  &\quad + i \int_{(\sigma)} L(s-1-z, \theta^2) \zeta(z) \frac{\Gamma(z) \Gamma(s-1-z)}{\Gamma(s-1)} \, dz,
  \end{align*}
when $\sigma > 1$ and $\Re s > \sigma$, where $L(s, \theta^2)$ is the normalized $L$-function
\begin{equation}\label{eq:Ls_theta_squared_identity}
  L(s,\theta^2) := \sum_{n \geq 1} \frac{r_2(n)}{n^{s}} = 4 \zeta_{\mathbb{Z}[i]}(s) =  4 \zeta(s)L(s,\chi),
\end{equation}
  and $\chi = (\tfrac{-4}{\cdot})$ is the non-trivial Dirichlet character of modulus $4$.
\end{proposition}

In Proposition~\ref{prop:DsPP_to_DsSS} and throughout the paper, we use the common notation
\begin{equation*}
  \frac{1}{2\pi i} \int_{(\sigma)} f(z) \, dz := \frac{1}{2\pi} \int_{-\infty}^\infty f(\sigma + it) \, dt.
\end{equation*}
The $\theta^2$ appearing in $L(s, \theta^2)$ refers to the square of the standard theta function
\begin{equation}\label{eq:standard_theta}
  \theta(z) = \sum_{n \in \mathbb{Z}} e^{2\pi i n^2 z}.
\end{equation}
Note that we maintain the notation of $L(s,\theta^2)$ rather than $4
\zeta(s)L(s,\chi)$ throughout this work to facilitate and encourage the
possible generalization of this construction to other non-cusp forms.

\begin{proof}
  Note that $P_2(n)^2$ and $S_2(n)^2$ are related by the formula
  \begin{equation*}
    P_2(n)^2 = S_2(n)^2 - 2 \pi n S_2(n) + \pi^2 n^2.
  \end{equation*}
  Divide by $n^s$, sum over $n \geq 1$, and simplify.
  For the middle term, note that
  \begin{align*}
    \sum_{n \geq 1} \frac{S_2(n)}{n^{s-1}} &= \sum_{n \geq 1} \frac{1 + r_2(n)}{n^{s - 1}} + \sum_{n \geq 1}\sum_{m=1}^{n-1} \frac{r_2(m)}{n^{s-1}} \\
    &= \zeta(s-1) + L(s-1, \theta^2) + \sum_{m,h \geq 1} \frac{r_2(m)}{(m+h)^{s-1}}.
  \end{align*}
  We decouple $m$ and $n$ in the final sum with the Mellin-Barnes identity
  \begin{equation*}
    \frac{1}{(m+h)^s} = \frac{1}{2\pi i} \int_{(\sigma)} \frac{1}{m^{s - z} h^z} \frac{\Gamma(z) \Gamma(s - z)}{\Gamma(s)} \, dz, \quad (\sigma > 0, \Re s > \sigma),
  \end{equation*}
  given in~\cite[6.422(3)]{GradshteynRyzhik07}.
  The remaining simplification is straightforward.


  %
  The identity~\eqref{eq:Ls_theta_squared_identity} follows immediately
  from comparing Euler products after noting that $r_2(n)/4 = \sum_{d \mid n}
  \chi(d)$ is multiplicative.

\end{proof}

\begin{remark}
  Simple factorizations for $L(s,\theta^k)$ are only known for $k=2,4,6,8$. Simplification along the lines of~\eqref{eq:Ls_theta_squared_identity} was therefore not available in the previous work~\cite{HulseGaussSphere} for general $k \geq 3$.
\end{remark}

As in~\cite[Proposition~2.2]{HulseGaussSphere} or~\cite[Proposition~3.1]{HulseKuanLowryDudaWalker17}, we can decompose $D(s, S_2 \times S_2)$ into the sum of a function $W_2(s)$ and an associated Mellin-Barnes integral.

\begin{proposition}\label{prop:DsSS_decomposition}
  The Dirichlet series associated to $S_2(n)^2$ decomposes into
\begin{align*}
  D(s,S_2 \times S_2) &= \zeta(s+2) + W_2(s) \\
    &\quad + \frac{1}{2\pi i} \int_{(\sigma)} W_2(s-z) \zeta(z) \frac{\Gamma(z)\Gamma(s+2-z)}{\Gamma(s+2)} \, dz
\end{align*}
for $\Re s > 2$ and $ 1< \sigma < \Re (s-1)$, in which
\begin{align*}
  W_2(s)
  &=
  \frac{16\zeta(s+2)^2L(s+2,\chi)^2}{(1+2^{-s-2})\zeta(2s+4)}
  +
  2 Z_2(s+2,0),
  \\
  Z_2(s,w)
  &=
  \sum_{h \geq 1} \sum_{n \geq 0}
  \frac{r_2(n+h)r_2(n)}{(n+h)^s h^w}.
\end{align*}
Here $Z_2(s,w)$ converges locally normally for $\Re s > 2$ and $\Re w \geq 0$.
\end{proposition}

\begin{proof}
	We first note the aesthetic result that
\begin{equation} \label{diagonal_closed_form}
  \sum_{n \geq 1} \frac{r_2(n)^2}{n^s}
  =
  \frac{16\zeta(s)^2L(s,\chi)^2}{(1+2^{-s})\zeta(2s)}.
\end{equation}
  Again using that $r_2(n)/4$ is multiplicative, this identity can be easily
  checked by comparing Euler products.
  Given this, the proof of~\cite[Proposition~2.2]{HulseGaussSphere} applies verbatim.
\end{proof}

In $W_2(s)$, the first term $\sum r_2(n)^2 n^{-s-2}$ has a double pole at $s = -1$, coming from the factor $\zeta^2(s)$ in the numerator.
This behavior is unique to the dimension $2$ case, as the rightmost pole of the analogous function, $\sum r_k(n)^2 n^{-s-k}$, is simple for all $k \geq 3$.

\section{Meromorphic Continuation of $D_2(s;h)$ and $Z_2(s,w)$}%
\label{sec:MeroCont}

In this section, we explain how to obtain the meromorphic continuations of the singly-summed shifted convolutions
\begin{equation*}
  D_2(s; h) := \sum_{n \geq 0} \frac{r_2(n+h) r_2(n)}{(n+h)^s},
\end{equation*}
as well as the doubly-summed shifted convolution
\begin{equation*}
  Z_2(s,w) := \sum_{h \geq 1} \sum_{n \geq 0} \frac{r_2(n+h) r_2(n)}{(n+h)^s h^w} = \sum_{h \geq 1} \frac{D_2(s; h)}{h^w}.
\end{equation*}

These constructions follow analogous work in~\cite{HoffsteinHulse13}
and~\cite{HulseGaussSphere}. More broadly, this analysis is based upon the
strategy of recognizing the series as an inner product with Poincar\'e series
and using the spectral expansion of the Poincar\'e series to understand the
shifted convolutions as a sum over a basis of automorphic forms; this strategy
has been used before to study shifted convolutions, including the appendix
of~\cite{sarnak2001estimates} and the work of Blomer and
Harcos~\cite{blomer2008spectral}, and others. Most of the previous work has
focused on shifted convolutions of cusp forms, and we describe the necessary
modifications here.

Further, a major distinction between the traditional Gauss circle problem and
the generalized Gauss circle problems in dimension $k \geq 3$ becomes apparent
in this section.

Let $P_h(z,s)$ denote the Poincar\'e series
\begin{equation*}
  P_h(z,s) = \sum_{\Gamma_\infty \backslash \Gamma_0(4)} \Im(\gamma z)^s e(h \gamma z),
\end{equation*}
where we use the common abbreviation $e(z) = \exp(2\pi i z)$.
Let $\theta(z)$ denote the standard theta function
as in~\eqref{eq:standard_theta}.
Note that $\theta(z)$
is a modular form of weight $\tfrac{1}{2}$ on $\Gamma_0(4) \backslash \mathcal{H}$.
A classic unfolding argument shows that for $\Re s$ sufficiently large,
\begin{equation}\label{eq:inner_product_initial}
\begin{split}
  \langle \lvert \theta^2 \rvert^2 \Im(\cdot), P_h(\cdot, \overline{s}) \rangle
  =\!&
  \int_{\Gamma_0(4) \backslash \mathcal{H}}\! \lvert \theta^2(z) \rvert^2 y \overline{P_h(z, \overline{s})} d\mu(z) 
  =
  \frac{\Gamma(s)D_2(s;h)}{(4\pi)^s},
\end{split}
\end{equation}
where $\langle \cdot, \cdot \rangle$ denotes the Petersson inner product on $\Gamma_0(4) \backslash \mathcal{H}$ and $d\mu(z) = dx\,dy/y^2$ is the corresponding Haar measure.
After dividing by $h^w$ and summing over $h$, one recovers $Z_2(s,w)$.

To understand the meromorphic properties of $D_2(s; h)$ and $Z_2(s,w)$, we perform a spectral expansion on $P_h(z,s)$.
However, it is not possible to immediately replace $P_h$ by its spectral expansion in the inner product because $\lvert \theta^2(z) \rvert^2 y \not \in L^2(\Gamma_0(4) \backslash \mathcal{H})$.
It is necessary to modify $\lvert \theta^2 \rvert^2 y$ to be square integrable.
In~\cite{HulseGaussSphere}, this was accomplished by subtracting appropriate Eisenstein series evaluated at specific parameters.
But in dimension $2$, the na\"{\i}ve choices of Eisenstein series would be evaluated at poles, so it is necessary to present a new approach.

\subsection{Modifying $\lvert \theta^2 \rvert^2 y$ to be Square Integrable}

Let $E_\mathfrak{a}(z,s)$ denote the Eisenstein series associated to the cusp $\mathfrak{a}$ of $\Gamma_0(4)\backslash \mathcal{H}$, given by
\begin{equation*}
  E_\mathfrak{a} (z,s) = \sum_{\gamma \in \Gamma_\mathfrak{a} \backslash \Gamma_0(4)} \Im(\sigma_\mathfrak{a}^{-1} \gamma z)^s,
\end{equation*}
where $\Gamma_\mathfrak{a} \subset \Gamma_0(4)$ is the stabilizer of the cusp $\mathfrak{a}$, and $\sigma_\mathfrak{a} \in \PSL_2(\mathbb{R})$ satisfies $\sigma_\mathfrak{a} \infty = \mathfrak{a}$ and induces an isomorphism $\Gamma_\mathfrak{a} \cong \Gamma_\infty$ through conjugation.
The quotient $\Gamma_0(4) \backslash \mathcal{H}$ has three cusps, which can be represented as $0, \frac{1}{2}$, and $\infty$.

The Eisenstein series $E_\mathfrak{a}(z,s)$ have Fourier expansions
around the cusp $\mathfrak{b}$
of the form
\begin{equation}\label{eq:Eisenstein_expansion}
\begin{split}
  E_\mathfrak{a}(\sigma_\mathfrak{b} z,s)
  &=
  \delta_{[\mathfrak{a}=\mathfrak{b}]} y^s
  +
  \pi^\frac{1}{2} \frac{\Gamma(s-\frac{1}{2})}{\Gamma(s)} \varphi_{\mathfrak{a}\mathfrak{b}0}(s) y^{1-s}
  \\
  &\quad + \frac{2\pi^s y^\frac{1}{2}}{\Gamma(s)}
  \sum_{n \neq 0}\varphi_{\mathfrak{a}\mathfrak{b} n}(s) \vert n \vert^{s-\frac{1}{2}} K_{s-\frac{1}{2}}(2\pi \vert n \vert y) e(nx),
\end{split}
\end{equation}
in which the coefficients $\varphi_{\mathfrak{a}\mathfrak{b}n}(s)$ are described in~\cite{DeshouillersIwaniec82}, for example.
Here and throughout, we use $\delta_{[\text{condition}]}$ as a Kronecker delta, which is $1$ if the condition is true and is otherwise $0$.

When $\mathfrak{b} = \infty$, we will write these coefficients as $\varphi_{\mathfrak{a} n}(s)$.
Then $\varphi_{\mathfrak{a}h}$ takes the form~\cite{DeshouillersIwaniec82}
\begin{align*}
  \varphi_{\mathfrak{a} h}(t) = \left(\frac{(v,4/v)}{4v}\right)^t \sum_{(\gamma, 4/v) =1}^\infty \gamma^{-2t} \hspace{-3 mm}\sum_{\substack{\delta (\gamma v)^*\\ \gamma\delta \equiv u \!\!\!\!\!\mod (v,4/v)}}\! e\left(\frac{h\delta}{\gamma v}\right).
\end{align*}
Straightforward computations (similar to those in~\cite[\S3.1]{Goldfeld06}) show that the coefficients are given by
\begin{equation}\label{eq:Eisen_coeffs}
\begin{split}
  \varphi_{0h}(t) &= \frac{\sigma^{(2)}_{1-2t}(h)}{4^t\zeta^{(2)}(2t)}, \qquad \varphi_{\frac{1}{2}h}(t) =  \frac{(-1)^h\sigma^{(2)}_{1-2t}(h)}{4^t\zeta^{(2)}(2t)}, \\
  \varphi_{\infty h}(t) &= \frac{2^{2-4t}\sigma_{1-2t}(\frac{h}{4}) - 2^{1-4t}\sigma_{1-2t}(\frac{h}{2})}{\zeta^{(2)}(2t)},
\end{split}
\end{equation}
when $h \neq 0$, where $\zeta^{(2)}(t)$ is the Riemann zeta function with its $2$-factor removed, $\sigma_\nu(h)$ is the sum of divisors function, and $\sigma_\nu^{(2)}(h)$ is the sum of odd-divisors function.
If $h/2$ is not an integer, the corresponding divisor sum $\sigma_\nu(h/2)$  is defined to be zero, and similarly for $h/4$.
\begin{lemma}
  Define $\mathcal{V}(z)$ by%
  \begin{equation}\label{eq:def_V}
    \mathcal{V}(z) = \lvert \theta^2(z) \rvert^2 \Im(z) - \Res_{u = 1}\frac{\left( E_\infty(z,u) +E_0(z,u)\right)}{u-1}.
  \end{equation}
  Then $\mathcal{V}(z) \in L^2(\Gamma_0(4) \backslash \mathcal{H})$.
\end{lemma}

The use of constant terms in Laurent expansions of Eisenstein series to modify the growth of functions at cusps is not new, and has been used for example in~\cite[\S6]{HulseKiralKuanLim16} and~\cite[\S5]{LowryDudaThesis} in a similar manner.

\begin{proof}

It is a classical result that as $y \to \infty$,
\begin{align*}
  \lvert \theta^2(z) \rvert^2 \Im(z) = \lvert \theta^2(\sigma_0 z) \rvert^2 \Im(\sigma_0 z)= y(1 + O(e^{-2\pi y})),
\end{align*}
and that $\theta(z)$ has exponential decay at the cusp $\tfrac{1}{2}$.
From the expansion~\eqref{eq:Eisenstein_expansion} and asymptotics of the $K$-Bessel function, we see that
\begin{equation}\label{eq:Eisenstein_expansion_simple}
  E_{\mathfrak{a}}(\sigma_{\mathfrak{b}}z, u)
  =
  \delta_{[\mathfrak{a} = \mathfrak{b}]} y^u + \pi^{\frac{1}{2}} \frac{\Gamma(u - \frac{1}{2})}{\Gamma(u)} \varphi_{\mathfrak{a} \mathfrak{b} 0}(u) y^{1-u} + O(e^{-2\pi y}).
\end{equation}

It is therefore natural to attempt to mollify the growth of $\lvert \theta^2(z) \rvert^2 y$ at the $0$ and $\infty$ cusps by subtracting $E_\infty(z,1)$ and $E_0(z,1)$, but both $E_\infty(z,u)$ and $E_0(z,u)$ have poles at $u = 1$.
In particular, $\varphi_{\mathfrak{a} 0}(u)$ has a simple pole at $u = 1$ in both cases.
Referring to~\eqref{eq:Eisenstein_expansion} and~\eqref{eq:Eisenstein_expansion_simple}, it is clear that ${\Res_{u = 1} (u-1)^{-1}E_\infty(z,u)}$  has leading term $y$, and secondary terms that are logarithmic and constant in $y$, and is otherwise of rapid decay (and similarly for $E_0$ with respect to the $0$ cusp).
As the constant terms of these Laurent expansions are modular, we conclude that
\begin{equation*}
  \mathcal{V}(z) := \lvert \theta^2(z) \rvert^2 y - \Res_{u = 1} \frac{E_\infty(z,u) +  E_0(z,u)}{u-1} \in L^2(\Gamma_0(4) \backslash \mathcal{H}),
\end{equation*}
which proves the lemma.
\end{proof}

\subsection{Modified Inner Product Representation}

We will use the modified function $\mathcal{V}(z)$ instead of $\lvert \theta^2(z) \rvert^2 y$ to study the meromorphic properties of $Z_2(s,w)$.
Replacing~\eqref{eq:inner_product_initial} with $\mathcal{V}$ shows that
\begin{equation*}
  \langle \mathcal{V}, P_h(\cdot, \overline{s}) \rangle = \frac{\Gamma(s)}{(4\pi)^s} D_2(s; h)
  -
  \left\langle \Res_{u = 1} \left(\frac{E_\infty(\cdot, u) + E_0(\cdot, u)}{u-1} \right), P_h(\cdot, \overline{s}) \right\rangle.
\end{equation*}
The inner product of the Eisenstein series against the Poincar\'{e} series can be directly computed (by unfolding and applying~\cite[6.621(3)]{GradshteynRyzhik07}) to be
\begin{equation}\label{eq:Eisenstein_against_Poincare}
  \langle E_\mathfrak{a}(\cdot, u), P_h(\cdot, \overline{s}) \rangle = \frac{2\pi^{u + \frac{1}{2}}}{(4\pi h)^{s - \frac{1}{2}}} h^{u - \frac{1}{2}} \varphi_{\mathfrak{a}h}(u) \frac{\Gamma(s + u - 1) \Gamma(s - u)}{\Gamma(s) \Gamma(u)},
\end{equation}
provided that $\Re s + u - 1 > 0$ and that $\Re u$ is sufficiently large.
The equality~\eqref{eq:Eisenstein_against_Poincare} may be subsequently extended by meromorphic continuation.
After some simplification, we have
\begin{equation*}
\begin{split}
  \left\langle \Res_{u = 1} \left(\frac{E_\infty(\cdot, u) + E_0(\cdot, u)}{u-1}\right), P_h(\cdot, \overline{s}) \right\rangle
  &=
  \Res_{u = 1} \frac{\langle E_\infty(\cdot, u) + E_0(\cdot, u), P_h(\cdot, \overline{s}) \rangle}{u-1} \\
  &= \frac{\pi \Gamma(s-1)}{(4\pi h)^{s-1}} \big(\varphi_{\infty h}(1) + \varphi_{0 h}(1) \big).
\end{split}
\end{equation*}
Here we have used that the coefficients $\varphi_{\mathfrak{a}h}(u)$ are holomorphic at $u = 1$ as long as $h \geq 1$, as can be seen from~\eqref{eq:Eisen_coeffs}.

This shows that
\begin{equation}\label{eq:Dsh_direct_exp}
  D_2(s;h)
  =
  \frac{(4\pi)^s}{\Gamma(s)} \langle \mathcal{V}, P_h(\cdot, \overline{s}) \rangle
  +
  \frac{4\pi^2}{s-1} \frac{\varphi_{\infty h}(1) + \varphi_{0 h}(1)}{h^{s-1}}.
\end{equation}
Dividing by $h^w$ and summing over $h \geq 1$ gives that
\begin{equation}\label{eq:Zsw_direct_exp}
  Z_2(s,w)
  =
  \frac{(4\pi)^s}{\Gamma(s)} \sum_{h \geq 1} \frac{\langle \mathcal{V}, P_h(\cdot, \overline{s}) \rangle}{h^w}
  +
  \frac{4\pi^2}{s-1} \sum_{h \geq 1} \frac{\varphi_{\infty h}(1) + \varphi_{0 h}(1)}{h^{s + w - 1}}.
\end{equation}

\begin{remark}
The difference between the expansion~\eqref{eq:Zsw_direct_exp} and its higher-dimensional analogue from equation~(3.7) in~\cite{HulseGaussSphere} is purely technical, and these expressions should be directly compared.
Indeed, the remainder of the description of the meromorphic properties of $Z_2$ is essentially the same as the description of $Z_k$ for $k \geq 3$, except at times when restricting to even dimension allows for greater simplification.
\end{remark}

\subsection{Spectral Expansion}

We now provide a spectral expansion of the Poincar\'e series $P_h(z,s)$ and insert this expansion into~\eqref{eq:Dsh_direct_exp} and~\eqref{eq:Zsw_direct_exp}.
Regarding $\mathcal{V}$ as a generic modular, square-integrable function, this is identical to the spectral expansion that appears in~\cite{HulseGaussSphere}.
We introduce the necessary notation to describe and state the spectral expansion, but we defer to~\cite[\S3.2]{HulseGaussSphere} for the proof.

The Poincar\'{e} series $P_h(z,s)$ has a spectral expansion (as given in~\cite[Theorem 15.5]{IwaniecKowalski04}) of the form \begin{align}
\begin{split} \label{eq:spectral_expansion_P}
    P_h(z,s) &= \sum_j \langle P_h(\cdot,s),\mu_j \rangle \mu_j(z)
    \\
    &\quad + \sum_\mathfrak{a} \frac{1}{4\pi}\int_{-\infty}^\infty
    \langle P_h(\cdot,s),E_\mathfrak{a}(\cdot,\tfrac{1}{2}+it)\rangle E_\mathfrak{a}(z,\tfrac{1}{2}+it)\, dt,
\end{split}
\end{align}
in which $\mathfrak{a}$ ranges over the three cusps of $\Gamma_0(4)\backslash \mathcal{H}$, and $\{\mu_j\}$ denotes an orthonormal basis
of the residual and cuspidal spaces, consisting of the constant form $\mu_0$ and of Hecke-Maass forms $\mu_j$ for $L^2(\Gamma_0(4) \backslash \mathcal{H})$ with associated types $\frac{1}{2}+it_j$. 
The inner product against the constant term $\mu_0$ vanishes, so we omit further consideration of it.
The Maass forms $\mu_j$ admit Fourier expansions
\begin{equation*}
  \mu_j(z) = \sum_{n \neq 0} \rho_j(n) y^{\frac{1}{2}} K_{it_j} (2\pi \lvert n \rvert y) e(nx),
\end{equation*}
where $e(x) = e^{2\pi i x}$, and have associated eigenvalues $\lambda_j(n)$ and $L$-functions
\begin{equation*}
  L(s, \mu_j) = \sum_{n \geq 1} \frac{\rho_j(n)}{n^s}.
\end{equation*}

The inner product $\langle P_h(\cdot,s), \mu_j \rangle$ decomposes mainly as a
product of gamma functions and $\rho_j(h)$, which has uniform exponential decay
in the $t_j$ aspect when $s$ is constrained to vertical strips.
Similarly, $\langle P_h(\cdot,s), E_\mathfrak{a}(\cdot,\tfrac{1}{2}+it)\rangle$
has uniform exponential decay in the $t$ aspect.
Thus the expansion converges locally uniformly.

Inserting the spectral expansion~\eqref{eq:spectral_expansion_P} into the
expression for $D_2(s;h)$ in~\eqref{eq:Dsh_direct_exp} and the expression for
$Z_2(s,w)$ in~\eqref{eq:Zsw_direct_exp} proves the following theorem.

\begin{theorem}

For $\Re s$ sufficiently large, the singly-summed shifted convolution $D_2(s;h)$ can be written as
  \begin{align}
  D_2(s;h) &= \frac{4 \pi^2}{s-1} \frac{\big(\varphi_{\infty h}(1) + \varphi_{0 h}(1)\big)}{h^{s-1}} \notag
  \\
  &\quad +
  \frac{2\pi}{h^{s-\frac{1}{2}}}
  \sum_j \rho_j(h) G(s, it_j)
  \langle \mathcal{V}, \mu_j \rangle \label{eq:D_spectral}
  \\
  &\quad + \sum_{\mathfrak{a}} \frac{1 }{i}
  \int_{(0)} \frac{\pi^{\frac{1}{2} + z} \overline{\varphi_{\mathfrak{a} h} (\frac{1}{2} - z)} G(s, z)}{h^{s - \frac{1}{2} - z} \Gamma(\frac{1}{2} + z)}
  \langle \mathcal{V}, E_\mathfrak{a}(\cdot, \tfrac{1}{2} - \overline{z}) \rangle \, dz, \notag
  \end{align}
and for $\Re w$ also sufficiently large, the doubly-summed shifted convolution $Z_2(s,w)$ can be written as
  \begin{align}
  Z_2(s,w)
  &=
  \frac{4\pi^2}{s-1} \sum_{h \geq 1} \frac{\big(\varphi_{\infty h}(1) + \varphi_{0 h}(1)\big)}{h^{w+s-1}} \notag
  \\
  &\quad + 2\pi \sum_j L(s+w-\tfrac{1}{2},\mu_j) G(s,it_j)
  \langle \mathcal{V},\mu_j \rangle \label{eq:Z_spectral} \\
  &\quad + \sum_{\mathfrak{a}} \frac{1}{i} \int_{(0)} \frac{G(s,z)\pi^{\frac{1}{2}+z}}{\Gamma(\frac{1}{2}+z)}
  \sum_{h \geq 1} \frac{\overline{\varphi_{\mathfrak{a}h}(\frac{1}{2}-z)}}{h^{s+w-z-\frac{1}{2}}}\langle \mathcal{V},E_\mathfrak{a}(\cdot,\tfrac{1}{2}-\overline{z})\rangle \, dz. \notag
  \end{align}

In both expressions, $G(s,z)$ denotes the collected gamma factors
\begin{equation*}
  G(s,z):= \frac{\Gamma(s-\frac{1}{2}+z)\Gamma(s-\frac{1}{2}-z)}{\Gamma(s)^2}.
\end{equation*}
We refer to first lines of~\eqref{eq:D_spectral} and~\eqref{eq:Z_spectral} as
the ``non-spectral part,'' to the second lines as the ``discrete part,'' and to
the third lines as the ``continuous part'' of the spectrum of $D_2$ or $Z_2$,
respectively.
\end{theorem}

Summing over $h$ has resulted in the appearance of the $L$-functions associated
to the Maass forms and Eisenstein series of the spectral expansion, and thus
yielded objects whose meromorphic properties are largely understood. This
phenomenon has been used before in the subconvexity results of Michel and
Harcos~\cite{michel2004subconvexity, harcos2006subconvexity}.

\subsection{Meromorphic Continuation of $D_2(s; h)$ and $Z_2(s,w)$}\label{ssec:D2_Z2_mero}

The description of the meromorphic continuation of $Z_2(s,w)$ can be obtained from the meromorphic continuation of $Z_k(s,w)$ as given in~\cite[\S3.3]{HulseGaussSphere} by specializing to $k = 2$, using the modified $\mathcal{V}$ as defined in~\eqref{eq:def_V}, and tracking changes in the non-spectral part.
The single shifted convolution $D_2(s;h)$ is described only implicitly there, so we
describe its properties explicitly and sketch the proofs here.

\begin{lemma}
  The single-sum shifted convolution $D_2(s;h)$ has meromorphic continuation to $\mathbb{C}$.
  The rightmost pole occurs at $s = 1$, coming from the non-spectral part.
  The function $D_2(s;h)$ is otherwise analytic in $\Re s > \tfrac{1}{2}$, though on the line $\Re s = \tfrac{1}{2}$ there is a line of poles appearing in the discrete part of the spectrum of $D_2$.
\end{lemma}

We consider the non-spectral, discrete, and continuous parts separately.
As there is only a single sum, the analysis is significantly simpler than the analysis of $Z_2(s,w)$.

\begin{proof}
  The meromorphic continuation of the non-spectral part of $D_2(s;h)$ is trivial, and we see a unique simple pole at $s=1$ with residue
  \begin{equation}\label{eq:D_residue_1}
    \Res_{s = 1} D_2(s;h) = 4 \pi^2 (\varphi_{\infty h}(1) + \varphi_{0 h}(1)).
  \end{equation}

  In the discrete part of the spectrum, there are poles at $s = \tfrac{1}{2} \pm it_j$ coming from the gamma factors in $G(s, it_j)$.
  As Selberg's Eigenvalue Conjecture is known for $\Gamma_0(4)$~\cite{Huxley85}, these poles lie in $\Re s \leq \frac{1}{2}$.
  Note that for any fixed $s$, the gamma factors $G(s, it_j)$ have exponential decay in $t_j$, so the sum converges absolutely.

The integrand of the continuous part of the spectrum has poles at $s = \tfrac{1}{2} \pm z$ due to the gamma factors in $G(s, z)$.
  Note that for any fixed $s$, the gamma factors $G(s,z)$ have exponential decay in $z$, so that the integral converges absolutely.
  Proving the meromorphic continuation of the continuous part of the spectrum is subtle, but the methodology of~\cite[\S4.4.2]{HulseKuanLowryDudaWalker17} or~\cite[\S3.3.3]{HulseGaussSphere} of iteratively shifting lines of integration and picking up residual terms applies here.
\end{proof}

\begin{remark}\label{remark:chamizo_compare}
  It is interesting to note that each individual $D_2(s;h)$ has poles at $s =
  \tfrac{1}{2} \pm it_j$ from the discrete spectrum, while the complete sum
  $Z_2(s, 0)$ does not.  That is, by averaging over $h$, the leading line of
  poles vanishes. This provides a spectral explanation for the vastly simpler
  behavior of averaged sums. The improved behavior of these averaged sums was
  also observed by Chamizo~\cite[\S4]{Chamizo1999}, but not in a way that
  Ivi\'{c} could leverage in~\cite{Ivic1996} while making the current best-known
  bound for~\eqref{eq:convolution_ivic}.
\end{remark}

We summarize the meromorphic behavior of $Z_2(s,w)$.
The function $Z_2(s,0)$ will be further analyzed in sections~\S\ref{sec:discrete_spectral_W}-\ref{sec:continuous_spectral_W}.

\begin{lemma}\label{lem:Z2_mero_summary}
  The doubly-summed shifted convolution $Z_2(s,w)$ has meromorphic continuation to $\mathbb{C}^2$.
  In particular, the specialized shifted convolution $Z_2(s,0)$ has meromorphic continuation to the plane.
  For $\Re s > -\tfrac{1}{2}$, all poles of $Z_2(s,0)$ come from the non-spectral part (which has a simple pole at $s = 2$ and a double pole at $s = 1$) and the continuous part of the spectrum (whose poles appear within the residual terms $\mathcal{R}^{\pm}_j$, as defined in~\eqref{eq:continuous_mero_with_resids}).
\end{lemma}

\subsubsection*{The non-spectral part.}

The non-spectral part can be described explicitly by computing the Dirichlet series associated to the coefficients $\varphi_{\mathfrak{a} h}(t)$.
Dividing by $h^w$ and summing over $h$ in~\eqref{eq:def_V}, we find that
\begin{equation}\label{eq:eisenstein_coeff_dirichlet_series}
  \begin{split}
    \sum_{h \geq 1} \frac{\varphi_{0h}(t)}{h^w} &= \frac{\zeta(w) \zeta^{(2)}(w-1+2t)}{4^t \zeta^{(2)}(2t)}, \\
    \sum_{h \geq 1} \frac{\varphi_{\frac{1}{2}h}(t)}{h^w} &= \frac{(2^{1-w}-1)\zeta(w)\zeta^{(2)}(w-1+2t)}{4^t \zeta^{(2)}(2t)},\\
    \sum_{h \geq 1} \frac{\varphi_{\infty h}(t)}{h^w} &= \frac{\zeta(w)\zeta(w-1+2t)}{2^{4t} \zeta^{(2)}(2t)} \left(\frac{1}{4^{w-1}}-\frac{1}{2^{w-1}} \right).
  \end{split}
\end{equation}
Thus the non-spectral part, as it appears in~\eqref{eq:Z_spectral}, can be written as
\begin{equation}\label{eq:nonspectral_sw_def}
  \frac{8 \zeta(s + w - 1) \zeta(s + w)}{(s-1)} (1 - 2^{1 - s - w} + 4^{1 - s - w}).
\end{equation}
This has clear meromorphic continuation to $\mathbb{C}^2$.
Specializing to $w = 0$, we note a simple pole at $s = 2$ and a double pole at $s = 1$.

\subsubsection*{The discrete spectrum}

The discrete part of the spectrum in~\eqref{eq:Z_spectral} has clear meromorphic continuation to the plane, coming from the meromorphic continuations of the $L$-functions $L(s, \mu_j)$ and the gamma functions.
Note that for any fixed $s$ away from poles, the gamma factor $G(s, it_j)$ has exponential decay in $t_j$ and the sum over $t_j$ converges absolutely.
Specializing to $w = 0$, we now analyze the poles.
The first line of apparent poles at $s = \tfrac{1}{2} \pm it_j$ do not actually occur.
For odd Maass forms $\mu_j$, the inner products $\langle \mathcal{V}, \mu_j \rangle$ vanish since $\mathcal{V}(z)$ is a linear combination of $\Im(z) |\theta^2(z)|^2$, an even form, and Eisenstein series at different cusps, both of which are orthogonal to $\mu_j$.
For even Maass forms $\mu_j$, the apparent poles are cancelled by trivial zeros of $L(s, \mu_j)$, as $L(-2m \pm it_j, \mu_j) = 0$ for any $m \in \mathbb{Z}_{\geq 0}$.
Thus the discrete part of the spectrum is analytic for $\Re s > -\tfrac{1}{2}$ and has poles at $s - \tfrac{1}{2} \pm it_j = -m$ for $m$ odd, $m \in \mathbb{Z}_{> 0}$.

\subsubsection*{The continuous spectrum}

The continuous part of the spectrum is the most nuanced.
For convenience, we rewrite the continuous component as
\begin{equation*}
  \frac{1}{i} \sum_{\mathfrak{a}}
  \int_{(0)} \frac{G(s,z)\pi^{\frac{1}{2}+z}}{\Gamma(\frac{1}{2}+z)}
  \zeta_\mathfrak{a}(s+w,z)
  \langle \mathcal{V},E_\mathfrak{a}(\cdot,\tfrac{1}{2}-\overline{z})\rangle \, dz,
\end{equation*}
in which $\zeta_\mathfrak{a}(s,z)$ is defined by
\begin{equation*}
  \zeta_\mathfrak{a}(s,z) = \sum_{h \geq 1} \frac{\overline{\varphi_{\mathfrak{a} h}(\frac{1}{2}-z)}}{h^{s-\frac{1}{2}-z}}.
\end{equation*}
We describe these Dirichlet series explicitly via~\eqref{eq:eisenstein_coeff_dirichlet_series} as
%
\begin{align*}
  \zeta_0(s,z&)= \frac{\zeta(s-\frac{1}{2}-z)\zeta^{(2)}(s-\frac{1}{2}+z)}{2^{1+2z} \zeta^{(2)}(1+2z)}, \\
  \zeta_\frac{1}{2}(s,z&) = \frac{\zeta(s-\frac{1}{2}-z)\zeta^{(2)}(s-\frac{1}{2}+z)}{2^{1+2z}\zeta^{(2)}(1+2z)}\left(\frac{2^{z}}{2^{s-\frac{3}{2}}}-1\right), \\
  \zeta_\infty(s,z&) =\frac{\zeta(s-\frac{1}{2}-z)\zeta(s-\frac{1}{2}+z)}{2^{2+4z}\zeta^{(2)}(1+2z)}\left(\frac{4^z}{4^{s-\frac{3}{2}}}-\frac{2^z}{2^{s-\frac{3}{2}}}\right).
\end{align*}

The integrand within the continuous component has apparent poles when $s + w - \tfrac{1}{2} \pm z = 1$ and when $s = \tfrac{1}{2} \pm z - j$ for $j \in \mathbb{Z}_{\geq 0}$.
In~\cite[\S3.3.3]{HulseGaussSphere}, it is proved that it is possible to meromorphically continue the continuous component past these apparent poles.
These apparent poles do not contribute poles at the expected locations, but instead introduce additional residual terms in the meromorphic continuation.
Overall, in the cases when $\Re(s+w) \neq \frac{3}{2}$ and $\Re(s) \neq \frac{1}{2}-j$, the meromorphic continuation of the continuous component is given by
\begin{align}\label{eq:continuous_mero_with_resids}
  &\frac{1}{i} \sum_{\mathfrak{a}}
  \int_{(0)} \frac{G(s,z)\pi^{\frac{1}{2}+z}}{\Gamma(\frac{1}{2}+z)}
  \zeta_\mathfrak{a}(s+w,z)
  \langle \mathcal{V},E_\mathfrak{a}(\cdot,\tfrac{1}{2}-\overline{z})\rangle \, dz \\
  &\quad + \delta_{[\Re(s+w) < \frac{3}{2}]}\big(\mathcal{R}^+_1(s,w) - \mathcal{R}^-_1(s,w) \big)
  + \!\sum_{j = 0}^{\lfloor \frac{1}{2} - \Re s \rfloor} \!\big( \mathcal{R}^+_{-j}(s,w) - \mathcal{R}^-_{-j}(s,w) \big).\notag
\end{align}
The terms $\mathcal{R}^+_1$ and $\mathcal{R}^-_1$ denote residual terms coming from apparent poles in the zeta functions in the continuous component.  These are given by
\begin{align*}
  \mathcal{R}_{1}^{\pm}(s,w):&= 2\pi  \!\Res_{\pm z=\frac{3}{2}-s-w}\! \frac{G(s,z) \pi^{\frac{1}{2}+z}}{\Gamma(\frac{1}{2}+z)} \sum_\mathfrak{a} \zeta_\mathfrak{a}(s+w,z)
  \left\langle \mathcal{V},E_\mathfrak{a}(\cdot,\tfrac{1}{2}-\overline{z})\right\rangle
\end{align*}
and, as described in~\eqref{eq:continuous_mero_with_resids}, only appear when $\Re(s + w) < \tfrac{3}{2}$.
The terms $\mathcal{R}^+_{-j}$ and $\mathcal{R}^-_{-j}$ denote residual terms coming from apparent poles from the gamma functions in the continuous component and are given by
\begin{align*}
  \mathcal{R}^\pm_{-j}(s,w) &= 2\pi  \sum_\mathfrak{a} \Res_{\pm z=\frac{1}{2}-j-s} \frac{G(s,z) \pi^{\frac{1}{2}+z}}{\Gamma(\frac{1}{2}+z)}\zeta_\mathfrak{a}(s+w,z) \left\langle \mathcal{V}, E_\mathfrak{a}(\cdot, \tfrac{1}{2}-\overline{z})\right\rangle.
\end{align*}
In the cases when $\Re(s+w) = \frac{3}{2}$ and $\Re(s)=\frac{1}{2}-j$,~\eqref{eq:continuous_mero_with_resids} is slightly altered, mainly that the line of integration for the integral term is bent slightly to the right into the zero-free region of $\zeta(1-2z)$, and we only have the corresponding $R^-$ residue for that line.

Note that for any fixed $s$ and $w$, only finitely many residual terms
$\mathcal{R}^\pm_{-j}$ appear in the meromorphic
continuation~\eqref{eq:continuous_mero_with_resids}.
As each residual term has meromorphic continuation to $\mathbb{C}^2$ (coming from the
meromorphic continuations of the zeta function, gamma function, and Eisenstein series), we
conclude that the continuous part of the spectrum admits meromorphic continuation to
$\mathbb{C}^2$.

\section{Analytic behavior of $W_2(s)$}\label{sec:W2}

Recall from Proposition~\ref{prop:DsSS_decomposition} that $W_2(s)$ is defined by
\begin{equation*}
  W_2(s)
  =
  \frac{16\zeta(s+2)^2L(s+2,\chi)^2}{(1+2^{-s-2})\zeta(2s+4)}
  +
  2 Z_2(s+2,0).
\end{equation*}
In this section, we will study the meromorphic properties of $W_2(s)$.
As described in section~\S\ref{ssec:D2_Z2_mero}, the discrete spectrum of $Z_2(s,0)$ has
infinitely many poles on the line $\Re s = - \frac{1}{2}$. We will also see that there are
poles in the continuous spectrum coming from zeros of the zeta function potentially near
this line.
Thus we focus our analysis of $W_2(s)$ to the half-plane $\Re s  > -\frac{5}{2}$.

Our analysis follows the decomposition of $W_2(s)$ into diagonal, non-spectral, discrete, and continuous parts.
When these observations are combined, we conclude the following theorem.

\begin{theorem}\label{thm:w2_properties}
  The function $W_2(s)$ is meromorphic in $\mathbb{C}$ and analytic in the right half-plane $\Re s > 0$.
  The rightmost pole of $W_2(s)$ occurs at $s=0$, with residue $2\pi^2$, coming from the non-spectral part.
  The function $W_2(s)$ is otherwise analytic in $\Re s > -\frac{5}{2}$, with the exception of a pole at $s=-\frac{3}{2}$ with residue
  \begin{equation*}
    \Res_{s=-\frac{3}{2}}
    W_2(s)
    =
    \frac{8(4-\sqrt{2})
    \zeta(\frac{3}{2})^2 L(\frac{3}{2},\chi)^2}{7 \pi^2 \zeta(3)}
    \approx 1.27046\; 77438,
  \end{equation*}
  coming from the continuous part of the spectrum.
\end{theorem}

\begin{remark}\label{remark:ref}
  It is useful to compare this Theorem to what can be gotten through classical
  estimates.
  Writing $W_2$ as
  \begin{align*}
    W_2(s) &= \sum_{n \geq 1} \frac{r_2(n)^2}{n^{s+2}} + 2\sum_{n, h \geq 1} \frac{r_2(n)r_2(n-h)}{n^{s+2}} \\
    &= \sum_{n \geq 1} \frac{r_2(n)}{n^{s+2}}\big(2r_2(0) + 2r_2(1) + \cdots + 2r_2(n-1) + r_2(n)\big)
  \end{align*}
  and using the classical Sierpi\'{n}ski estimate $\sum_{n \leq R} r_2(n) = 2
  \pi R + O(R^{1/3})$, we have that
  \begin{equation*}
    W_2(s) = \sum_{n \geq 1} \frac{2 \pi r_2(n)}{n^{s+1}} + H(s) = 8 \pi \zeta(s+1) L(s+1, \chi) + H(s),
  \end{equation*}
  where $H(s)$ is analytic for $\Re s > - \frac{2}{3}$.
  Note that we have rewritten the Dirichlet series
  using~\eqref{eq:Ls_theta_squared_identity}.

  From this classical point of view, the leading pole with residue $2 \pi^2$
  and analytic continuation to $\Re s > - \frac{2}{3}$ are both clear.
  Thus the new content of Theorem~\ref{thm:w2_properties} is the meromorphic continuation and the analysis of the pole at $s = -\frac{3}{2}$.
  In fact, the pole at $s = -\frac{3}{2}$ ultimately leads to the leading pole of $D(s, P_2 \times P_2)$.
\end{remark}

\subsection{The Diagonal Part}

We first consider the first term in $W_2(s)$,
\begin{equation}\label{eq:diagonal_base}
 \frac{16\zeta(s+2)^2L(s+2,\chi)^2}{(1+2^{-s-2})\zeta(2s+4)},
\end{equation}
which we call the ``diagonal part.''
Using well-known properties of $\zeta(s)$ and $L(s,\chi)$, we see that the diagonal part is analytic in the right half-plane $\Re s > -1$.
There is a double pole at $s = -1$ coming from $\zeta(s+2)^2$ with principal part
\[
  \frac{4}{(s+1)^2}
  +
  \frac{8\gamma + \frac{4}{3}\log 2 +\frac{32}{\pi}L'(1,\chi)-\frac{48}{\pi^2}\zeta'(2)}
  {s+1},
\]
in which we have used the evaluation $L(1,\chi)=\pi/4$ to simplify.

The diagonal part has infinitely many simple poles on the line $\Re s = -2$, coming from $1+2^{-s-2}=0$ as well as infinitely many poles at the zeros of $\zeta(2s+4)$.
Note that $\zeta(2s+4)^{-1}$ is analytic for $\Re s > -\frac{3}{2}$.

\begin{remark}\label{rem:annihilation_remark}
  As in~\cite{HulseKuanLowryDudaWalker17, HulseGaussSphere}, the diagonal part perfectly cancels with a pair of residual terms from the continuous spectrum once $\Re s < - \frac{3}{2}$.
  Thus the poles coming from zeros of $(1+2^{-s-2})\zeta(2s+4)$ will not affect our analysis of $W_2(s)$.
\end{remark}

\subsection{The Non-Spectral Part}

As shown in~\eqref{eq:nonspectral_sw_def}, the non-spectral part of $W_2(s)$ is given by
\begin{equation*}
  \mathfrak{E}_2(s) := \frac{16\zeta(s+1)\zeta(s+2)}{s+1}(1-2^{-s-1}+4^{-s-1}).
\end{equation*}
The meromorphic behavior of $\mathfrak{E}_2(s)$ is determined by the behavior of $\zeta(s)$.
This term has a simple pole at $s=0$ with residue $2\pi^2$, a double pole at $s=-1$ with principal part
\begin{equation*}
  -\frac{8}{(s+1)^2} - \frac{8(\gamma+\log \pi)}{s+1},
\end{equation*}
and is otherwise analytic.

\subsection{The Discrete Spectral Part}\label{sec:discrete_spectral_W}

The discrete part of $W_2(s)$ is analytic for $\Re s > -\frac{5}{2}$, where we focus our analysis.
On the line $\Re s = -\frac{5}{2}$, the discrete part has a line of poles at $s
= -\frac{5}{2} \pm it_j$, where $\frac{1}{4} +t_j^2$ denotes a discrete
eigenvalue of the Laplace-Beltrami operator on $\Gamma_0(4) \backslash
\mathcal{H}$.
Further, on each line $\Re s = -\frac{5}{2} - n$ where $n \in \mathbb{N}_{\geq
1}$, the discrete part has another line of poles.

\begin{remark}\label{rmk:discrete}

We note that the inner product $\langle \mathcal{V}, \mu_j \rangle$ factors as
\begin{equation}\label{eq:inner_product_factorization}
  \frac{2}{L(1, \chi)}
  \rho_j(1)
  \widetilde{L}(\tfrac{1}{2}, \overline{\mu_j})
  \widetilde{L}(\tfrac{1}{2}, \overline{\mu_j} \times \chi)
  \Gamma(\tfrac{1}{2} + it_j)
  \Gamma(\tfrac{1}{2} - it_j)
\end{equation}
in dimension $k=2$, where $\chi$ is the nontrivial character mod $4$, as before, and
\begin{equation*}
  \widetilde{L}(s,\overline{\mu_j} \times \chi)
  =
  \sum_{n \geq 1}
  \frac{\chi(n) \overline{\lambda_j(n)}}
       {n^s}.
\end{equation*}
This identity follows from the observation that $y^{\frac{1}{2}}\theta^2(z) =
E^1_\infty(z, \frac{1}{2})$, where $E^1_\infty(z, s)$ is the weight-one Eisenstein series
for the cusp at infinity, and thus the inner product becomes a special value of a
Rankin-Selberg convolution,
\begin{equation*}
  \frac{L(s,\theta^2 \times \overline{\mu_j} )}{L(2s,\chi)}
  =
  \sum_{n\geq 1}
  \frac{r_2(n)\overline{\rho_j(n)}}{n^s},
\end{equation*}
which we then factor into lower-degree $L$-functions by comparing Euler factors. A similar
construction can be obtained in the $k=4$ case.

Along with classical estimates for $\rho_j(1)$ and gamma functions, this explicit
description of $\langle \mathcal{V}, \mu_j \rangle$ is enough to permit analysis beyond
the line of poles of the discrete spectrum.

\end{remark}

\subsection{The Continuous Spectral Part}\label{sec:continuous_spectral_W}

As shown in section~\S\ref{ssec:D2_Z2_mero}, infinitely many residual terms $\mathcal{R}^\pm_{-j}(s,w)$ appear in the meromorphic continuation of the continuous part of $Z_2(s,w)$.  However, the only residual terms that appear in the half-plane $\Re s > - \frac{5}{2}$ are $\mathcal{R}^{\pm}_{1}$ and $\mathcal{R}^{\pm}_{0}$.

In~\cite[Lemma~4.3]{HulseGaussSphere}, it is shown that
\begin{equation}\label{eq:residuals_antisymmetric}
  \mathcal{R}^{+}_{1}(s,0) = - \mathcal{R}^{-}_{1}(s,0).
\end{equation}
The proof applies in the case $k = 2$ as well, and~\eqref{eq:residuals_antisymmetric} shows that the total contribution of $\mathcal{R}^{+}_{1} - \mathcal{R}^{-}_{1}$ within $2Z_2(s+2, 0)$ is given by
\begin{equation*}
  4 \mathcal{R}^{+}_{1}(s+2, 0) = \frac{4 \Gamma(s+\frac{3}{2}) \pi^{s+\frac{3}{2}} \langle \mathcal{V}, E_0(\cdot ,- \overline{s}) \rangle}{\Gamma(s+2)^2}.
\end{equation*}

We relate the inner product $\langle \mathcal{V}, E_0(\cdot, -\overline{s}) \rangle$ to the diagonal part through Gupta's generalization~\cite{Gupta00} of Zagier's regularized Rankin--Selberg construction~\cite{ZagierRankinSelberg}. 
As noted in~\cite[\S4.1]{HulseGaussSphere}, following the
regularization technique of unfolding and analytic continuation of Gupta and Zagier leads to the equality
\begin{equation*}
  \langle \mathcal{V}, E_0(\cdot, \overline{s}) \rangle
  =
  \frac{\Gamma(s)}{(4\pi)^s } \sum_{m \geq 1} \frac{r_2(m)^2}{m^s}
  =
  \frac{16\Gamma(s)\zeta(s)^2 L(s,\chi)^2}{(4\pi)^s (1+2^{-s})\zeta(2s)},
\end{equation*}
valid initially for $0 < \Re s < 1$ and extended through analytic continuation.
Note that we have used the identity~\eqref{diagonal_closed_form}
for the second equality.
It follows that
\begin{equation*}
  4 \mathcal{R}_1^+(s+2,0)
  = \frac{4^{s+3}\pi^{2s+\frac{3}{2}}\Gamma(s+\frac{3}{2}) \Gamma(-s)\zeta(-s)^2 L(-s,\chi)^2}{\Gamma(s+2)^2(1+2^s)\zeta(-2s)}.
\end{equation*}

Recall from section~\S\ref{ssec:D2_Z2_mero} that the residual terms $\mathcal{R}^{\pm}_{1}(s+2, 0)$ only contribute when $\Re s < -\frac{1}{2}$.
It therefore suffices to study $\mathcal{R}_1^\pm(s+2,0)$ in the strip $\Re s  \in (-\frac{5}{2},-\frac{1}{2})$.
In this region, the only poles come from $\Gamma(s+\frac{3}{2})$ and $\zeta(-s)$.
There is a double pole at $s=-1$ with principal part
\begin{equation*}
  \frac{4}{(s+1)^2}  + \frac{24 \log \pi -4\log 2 + 144\zeta'(2) /\pi^2}{3(s+1)} -\frac{32 L'(1,\chi)}{\pi(s+1)},
\end{equation*}
as well as a simple pole at $s = -\frac{3}{2}$, coming from $\Gamma(s + \frac{3}{2})$, with residue
\begin{equation*}
  \frac{8(4-\sqrt{2}) \zeta(\frac{3}{2})^2 L(\frac{3}{2},\chi)^2}{7 \pi^2 \zeta(3)}.
\end{equation*}

The next pair of residual terms also satisfy $\mathcal{R}_{0}^{+}(s,0) = - \mathcal{R}_{0}^{-}(s,0)$.
We compute that the total contribution of $\mathcal{R}^{+}_{0} - \mathcal{R}^{-}_{0}$ within
$2 Z_2(s+2, 0)$ is given by
\begin{equation*}
  4\mathcal{R}_0^+(s+2,0)
  =
  - \frac{(4\pi)^{s+2}}{\Gamma(s+2)}
  \left\langle \mathcal{V}, E_\infty(\cdot,\overline{s}+2)\right\rangle
  =
  - \frac{16\zeta^2(s+2)L(s+2,\chi)^2}{(1+2^{-s-2})\zeta(2s+4)}.
\end{equation*}
Thus the contribution from $\mathcal{R}_0^{\pm}$ \emph{exactly cancels} with the diagonal
part~\eqref{eq:diagonal_base} in the left half-plane $\Re s < -\frac{3}{2}$, as stated in
Remark~\ref{rem:annihilation_remark}.

\subsection*{Understanding the continuous part to the left of $\Re s = -\frac{5}{2}$}

To study the meromorphic properties of $W_2(s)$ beyond the line
$\Re s = - \frac{5}{2}$, it is necessary to study the behavior of
$\mathcal{R}_{1}^\pm(s+2, 0)$ further and to study the additional residual terms
$\mathcal{R}_{-1}^\pm(s+2, 0)$.

There is a pole at $s = -5/2$ coming from $\Gamma(s + \frac{3}{2})$ in
$\mathcal{R}_{1}^\pm(s+2, 0)$ with a residue easily understood in terms of
special values of $L$-functions. The remaining poles of $\mathcal{R}_1^\pm$ all
come from this gamma function.

But the residual terms $\mathcal{R}_{-1}^\pm(s+2, 0)$ are much more troublesome.
These residual terms appear in the meromorphic continuation when $\Re s < -\tfrac{5}{2}$.
In contrast to the previous residual terms, it is not true that
$\mathcal{R}_{-1}^+(s, 0) = - \mathcal{R}_{-1}^-(s, 0)$.
Further, these terms do not seem to cancel with other parts of $W_2(s)$.

The term $\mathcal{R}_{-1}^+(s+2, 0)$ is given by
\begin{equation*}
  -2\pi \sum_{\mathfrak{a}}
  \frac{\Gamma(2s + 4)}{\Gamma(s+2)^2\Gamma(-2-s)}
  \pi^{-s}
  \zeta_{\mathfrak{a}}(s+2, -\tfrac{5}{2} - s)
  \langle
    \mathcal{V},
    E_\mathfrak{a}(\cdot, 3 + \overline{s})
  \rangle,
\end{equation*}
and the corresponding term $\mathcal{R}_{-1}^-(s+2, 0)$ is similar.
The Eisenstein series $E_\mathfrak{a}(\cdot, s)$ can have poles
at nontrivial zeroes of the zeta function $\zeta(2s)$. Thus
$\mathcal{R}_{-1}^+(s+2, 0)$ has potential poles at $s = -3 + \tfrac{\rho}{2}$,
where $\rho$ is a zero of $\zeta(s)$.

Thus there are potential poles just to the left of the line $\Re s = - 5/2$,
coming from zeros of the zeta function. If we assume the Riemann hypothesis,
so that $\Re \rho = 1/2$, then there are no potential poles until
$\Re s = -2.75$.

\section{Analysis of $D(s, P_2 \times P_2)$}\label{sec:DsP2P2}

In this section we begin our study of $D(s,P_2 \times P_2)$, with an emphasis on the behavior of its leading poles.
By analogy with $D(s, P_k \times P_k)$ in dimensions $k \geq 3$, one should expect a large amount of cancellation in the rightmost poles and residues of the components of $D(s, P_2 \times P_2)$.

Combining Proposition~\ref{prop:DsPP_to_DsSS}, which relates $D(s, P_2 \times P_2)$ and $D(s, S_2 \times S_2)$, with Proposition~\ref{prop:DsSS_decomposition}, which relates $D(s, S_2 \times S_2)$ to $\zeta(s)$ and $W_2(s)$, yields the following unified expression for $D(s, P_2 \times P_2)$:
\begin{align}
  D(s, P_2 \times P_2)
  &=
  \zeta(s)+ W_2(s-2)+\pi^2 \zeta(s-2) \label{line1:p2_decomposition}
  \\
  &\quad -  2\pi\zeta(s-1) -2\pi L(s-1,\theta^2) \label{line2:p2_decomposition}
  \\
  &\quad + \frac{1}{2\pi i} \int_{(\sigma)} W_2(s-2-z) \zeta(z) \frac{\Gamma(z)\Gamma(s-z)}{\Gamma(s)}\, dz \label{line3:p2_decomposition}
  \\
  &\quad + i \int_{(\sigma)} L(s-1-z,\theta^2)\zeta(z) \frac{\Gamma(z)\Gamma(s-1-z)}{\Gamma(s-1)}\, dz, \label{line4:p2_decomposition}
\end{align}
initially valid for $\Re s \gg 1$ and $\sigma \in (1,\Re s)$.
We restrict our analysis of $D(s, P_2 \times P_2)$ to the half-plane $\Re s > \frac{1}{2}$ so as to avoid a line of poles appearing in the discrete part of the spectrum of~\eqref{line3:p2_decomposition}.
For each line~\eqref{line1:p2_decomposition}--\eqref{line4:p2_decomposition}, we study the locations and residues of poles for $\Re s > \frac{1}{2}$.
This information is collected in Table~\ref{table} for easy reference.

\begin{table}[t]
\caption{Summary of Polar Data in the Half-Plane $\Re s > \frac{1}{2}$}\label{table}
\small
\renewcommand{\arraystretch}{1.5}
\begin{threeparttable}
\begin{tabular}{llll}
  \toprule

  \textsc{pole location} & \textsc{line} & \textsc{contributing term} & \textsc{residue} \\

  \midrule

	$s=3$	&\eqref{line1:p2_decomposition}	& $\pi^2\zeta(s-2)$ & $\pi^2$ \\
	$s=3$	&\eqref{line3:p2_decomposition}	& $\frac{\mathfrak{E}_2(s-3)}{s-1}$, from $\frac{W_2(s-3)}{s-1}$ & $\pi^2$ \\
	$s=3$	&\eqref{line4:p2_decomposition}	& $-2\pi \frac{L(s-2,\theta^2)}{s-2}$	& $-2\pi^2$ \\

  \midrule

	$s=2$	&\eqref{line1:p2_decomposition} & $\mathfrak{E}_2(s-2)$, from $W_2(s-2)$ & $2\pi^2$ \\
	$s=2$	&\eqref{line2:p2_decomposition}	& $-2\pi \zeta(s-1)$		& $-2\pi$  \\
	$s=2$	&\eqref{line2:p2_decomposition}	& $-2\pi L(s-1,\theta^2)$	& $-2\pi^2$ \\
	$s=2$	&\eqref{line3:p2_decomposition}	& $-\frac{\mathfrak{E}_2(s-2)}{2}$, from $-\frac{W_2(s-2)}{2}$ & $-\pi^2$ \\
  $s=2$	&\eqref{line4:p2_decomposition}	& $\pi L(s-1,\theta^2)$ & $\pi^2$ \\
  $s=2$	&\eqref{line4:p2_decomposition}	& $-2\pi \frac{L(s-2,\theta^2)}{s-2}$	& $2\pi$ \\

  \midrule

	$s= \frac{3}{2}$	&\eqref{line3:p2_decomposition}	& $\frac{4\mathcal{R}_1^+(s-1,0)}{s-1}$, from $\frac{W_2(s-3)}{s-1}$ &  $\frac{16(4-\sqrt{2}) \zeta(\frac{3}{2})^2 L(\frac{3}{2},\chi)^2}{7 \pi^2 \zeta(3)}$ \\
	\midrule

	$s=1$	&\eqref{line1:p2_decomposition} & $\zeta(s)$		&	$1$ \\
	$s=1$	&\eqref{line3:p2_decomposition}	& $\frac{W_2(s-3)}{s-1}$ &	$W_2(-2)$ \\
	$s=1$	&\eqref{line3:p2_decomposition}	& $\frac{s\mathfrak{E}_2(s-1)}{12}$ & $\frac{\pi^2}{6}$ \\

  \bottomrule
\end{tabular}
\end{threeparttable}
\end{table}

\subsection*{Poles from terms in~\eqref{line1:p2_decomposition} and~\eqref{line2:p2_decomposition}}

These two lines contain simple $L$-functions and $W_2(s-2)$, so our polar data is either classically known or given by Theorem~\ref{thm:w2_properties}.

\subsection*{Poles from terms in~\eqref{line3:p2_decomposition}}

To understand the meromorphic properties of the integral, shift the line of integration $(\sigma)$ left to $(-3+\epsilon)$ for some small $\epsilon > 0$.
There are poles at $z = 1$ from $\zeta(z)$ as well as poles at $z = 0$ and $z = -1$ from $\Gamma(z)$.
By Cauchy's residue theorem, line~\eqref{line3:p2_decomposition} can be written as
\begin{equation*}
\begin{split}
  \frac{1}{2\pi i} \int_{(-3 + \epsilon)} &W_2(s - 2 - z) \zeta(z) \frac{\Gamma(z) \Gamma(s  - z)}{\Gamma(s)} \, dz
  \\
  + &\frac{W_2(s - 3)}{s -1} - \frac{W_2(s - 2)}{2} + \frac{sW_2(s - 1)}{12}.
\end{split}
\end{equation*}
The shifted integral is analytic in the right half-plane $\Re s > -1+\epsilon$, and the poles of the extracted residue terms can be understood from the poles of $W_2(s)$ as described in Theorem~\ref{thm:w2_properties}.

\subsection*{Poles from terms in~\eqref{line4:p2_decomposition}}

As above, shift the line of integration $(\sigma)$ to $(-3 + \epsilon)$ to show that the integral in~\eqref{line4:p2_decomposition} is given by
\begin{equation*}
\begin{split}
  i \int_{(-3 + \epsilon)} &L(s - 1 - z, \theta^2) \zeta(z) \frac{\Gamma(z) \Gamma(s -1 - z)}{\Gamma(s -1)} \, dz
  \\
  -2\pi \bigg( &\frac{L(s - 2, \theta^2)}{s - 2} - \frac{L(s - 1, \theta^2)}{2} + \frac{L(s, \theta^2)(s -1)}{12}\bigg).
\end{split}
\end{equation*}
The shifted integral is analytic for $\Re s> -1+\epsilon$, and the poles of the $z$-residues can be understood using the identity $L(s, \theta^2) = 4 \zeta(s) L(s, \chi)$ noted in~\eqref{eq:Ls_theta_squared_identity}.

\subsection{Examination of Poles and Their Cancellation}

With reference to Table~\ref{table}, we see that the residues of the poles at $s = 3$ cancel, so that $D(s, P_2 \times P_2)$ is analytic for $\Re s > 2$.
To examine the potential pole at $s=2$, we compute
\begin{equation}\label{eq:L0theta_eval}
  -2\pi L(0,\theta^2) = -8\pi \zeta(0)L(0,\chi)  = 2\pi,
\end{equation}
in which we've used that $\zeta(0)=-1/2$ and that $L(0,\chi)=1/2$.
Referring to Table~\ref{table}, we see that the residues of the poles at $s = 2$ cancel as well.

The pole at $s=\frac{3}{2}$ clearly does not cancel, and represents the leading pole of $D(s,P_2 \times P_2)$.

To understand the residue at $s=1$, we must compute $W_2(-2)$.
This calculation is simplified by the observation that $\mathcal{R}^\pm_0$
perfectly cancels with the diagonal part in this region, so both can be ignored.
The contribution from the non-spectral term is $\mathfrak{E}_2(-2) = -2$.
The contribution from $\mathcal{R}_1^\pm$ is $4\mathcal{R}_1^+(0,0)$, which
vanishes since $\mathcal{R}_1^\pm$ has $\Gamma(s+2)^2$ in its denominator.
Similarly, the discrete and continuous spectral terms vanish because they have
factors of $G(0,it_j)=0$ and $G(0,z)=0$ in them, respectively.
Thus $W_2(2) = -2$, and it follows that the pole at $D(s, P_2 \times P_2)$ at $s
= 1$ has residue $\frac{\pi^2}{6}-1$.

On the line $\Re s = 1/2$, the term $W_2(s-3)/(s-1)$ has a line of poles coming
from the discrete spectrum. In addition, there are potential poles from the
continuous spectrum at $s = 1/2$ and $s = \rho/2$ for each nontrivial zero
$\rho$ of $\zeta(s)$.

From these observations we derive the following theorem.

\begin{theorem}\label{thm:DsP2xP2_properties}
The Dirichlet series $D(s, P_2 \times P_2)$, originally defined in the right half-plane $\Re s > 3$ by the series
\begin{equation*}
  \sum_{m =1}^\infty \frac{P_2(m)^2}{m^{s}},
\end{equation*}
has meromorphic continuation to $\mathbb{C}$ given by~\eqref{line1:p2_decomposition}--\eqref{line4:p2_decomposition}.
It is analytic in the right half-plane $\Re s > \frac{3}{2}$ and has a pole at $s=\frac{3}{2}$ with residue
\begin{equation*}
  C_\frac{3}{2}:=\frac{16(4-\sqrt{2}) \zeta(\frac{3}{2})^2 L(\frac{3}{2},\chi)^2}{7 \pi^2 \zeta(3)}.
\end{equation*}
The function $D(s,P_2 \times P_2)$ has a second simple pole at $s=1$ with residue $\frac{\pi^2}{6}-1$ and is otherwise analytic in the right half-plane $\Re s > \frac{1}{2}$.
\end{theorem}

\begin{corollary}
  The Dirichlet series $D(s, S_2 \times S_2)$ has meromorphic continuation to the plane, attainable from Theorem~\ref{thm:DsP2xP2_properties} and Proposition~\ref{prop:DsPP_to_DsSS}.
\end{corollary}

\begin{remark}
Much of the analysis of $D(s,P_k \times P_k)$ in~\cite{HulseGaussSphere} carries over to $D(s,P_2 \times P_2)$, which makes it possible to identify key differences in the meromorphic behavior of $D(s,P_k \times P_k)$ between the cases $k=2$ and $k \geq 3$.  Notably, the leading pole at $s=\frac{3}{2}$ in the dimension $2$ case corresponds to a ``traveling pole'' at $s= \frac{5-k}{2}$ in dimension $k$ which contributes to the rightmost pole at $s=1$ in dimension $3$ and is otherwise non-dominant.

Movement of this pole relative to a fixed pole at $s=1$ accounts for the apparent phase change in the generalized Gauss circle problem between dimensions $k=2$ and $k \geq 3$.
\end{remark}

\section{Second Moment Analysis}\label{sec:second_moment}

In this section, we produce estimates for the discrete Laplace transform $\sum_{n \geq 1} P_2(n)^2 e^{-n/X}$ and the continuous Laplace transform $\int_0^\infty P_2(t)^2 e^{-t/X} \, dt$.
We do this by estimating the integral
\begin{equation*}
  \frac{1}{2\pi i} \int_{(4)} D(s, P_2 \times P_2) X^s \Gamma(s) ds = \sum_{n \geq 1} P_2(n)^2 e^{-n/X}
\end{equation*}
using the meromorphic information from Theorem~\ref{thm:DsP2xP2_properties}.

\begin{theorem}\label{thm:smooth_cutoff}
  We have
  \begin{equation*}
    \sum_{n \geq 1} P_2(n)^2 e^{-n/X}
    =
    C_\frac{3}{2} \Gamma(\tfrac{3}{2}) X^\frac{3}{2}
    +
    \Big(\frac{\pi^2}{6}-1\Big) X
    +
    O_\epsilon\big(X^{\frac{1}{2}+\epsilon}\big)
  \end{equation*}
  for any $\epsilon > 0$, in which $C_\frac{3}{2}$ is the constant defined in Theorem~\ref{thm:DsP2xP2_properties}.
\end{theorem}

\begin{proof}
  The proof of~\cite[Theorem~6.3]{HulseGaussSphere} for dimensions $k \geq 3$
  applies, mutatis mutandis, in the dimension $k = 2$ case.
  Briefly, after making the necessary modification to $\mathcal{V}$ as
  in~\eqref{eq:def_V} and studying the meromorphic continuation in
  Theorem~\ref{thm:w2_properties}, we have that $W_2(s)$ is of polynomial
  growth in vertical strips.
  Then~\cite[Lemma~6.2]{HulseGaussSphere} shows that the Mellin-Barnes integral
  appearing in the decomposition of $D(s, S_2 \times S_2)$ from
  Proposition~\ref{prop:DsSS_decomposition} and in the decomposition of $D(s,
  P_2 \times P_2)$ from Theorem~\ref{thm:DsP2xP2_properties} is also of
  polynomial growth, so that $D(s, S_2 \times S_2)$ and $D(s, P_2 \times
  P_2)$ have polynomial growth in vertical strips.
  Once polynomial growth is known, the theorem follows from a routine
  computation in shifting lines of integration.
  In particular, it suffices to shift the line of integration to $\Re s =
  \frac{1}{2} + \epsilon$ and account for the residues stated in
  Theorem~\ref{thm:DsP2xP2_properties}.
\end{proof}

Assuming the Riemann hypothesis, we can give more precise information about the
error term. Shifting the line of integration to $\Re s = \frac{1}{4} + \epsilon$
now passes an additional pole at $s = 1/2$ from $\mathcal{R}_1^\pm$ in the
continuous spectrum and at $s = \tfrac{1}{2} \pm it_j$ from the discrete
spectrum.

From the expansion~\eqref{eq:inner_product_factorization} for $\langle
\mathcal{V}, \mu_j\rangle$, and recalling that $L(1, \chi) = \tfrac{\pi}{4}$, we
can write the sum of the residues at $s = \tfrac{1}{2} \pm it_j$ as
\begin{equation}\label{eq:gross_residues}
\begin{split}
  - 32
  \sum_{\pm t_j}
  &\widetilde{L}(-1 \pm it_j, \mu_j)
  \widetilde{L}(\tfrac{1}{2}, \overline{\mu_j})
  \widetilde{L}(\tfrac{1}{2}, \overline{\mu_j} \otimes \chi)
  \frac{\rho_j(1)^2}{-\frac{1}{2} \pm it_j}
  \times
  \\
  &
  \frac{\Gamma(-1 \pm 2 it_j)
        \Gamma(\tfrac{1}{2} + it_j)
        \Gamma(\tfrac{1}{2} - it_j)
        }
       {\Gamma(-\frac{3}{2} \pm it_j)^2}
  \Gamma(\tfrac{1}{2} \pm it_j)
  X^{\frac{1}{2} \pm it_j}
  ,
\end{split}
\end{equation}
where we use $\widetilde{L}(\cdot, \mu_j)$ to denote the normalized $L$-function
as in Remark~\ref{rmk:discrete}.
By~\cite[\S4]{HoffsteinHulse13} and Stirling's forumla, $\rho_j(1)^2 \Gamma(1
\pm 2it_j)$ has polynomial growth on average over $t_j$.
The exponential growth from the two gamma functions in the denominator cancel
with the exponential decay of two gamma functions in the numerator, and thus the
remaining gamma function guarantees exponential decay of the summands.
Taking absolute values and bounding, we see that this sum converges absolutely.

This proves the following theorem.

\begin{theorem}\label{thm:optimal_error}
  Assume the Riemann hypothesis. Then we have
  \begin{equation*}
  \begin{split}
    \sum_{n \geq 1} P_2(n)^2 e^{-n/X}
    &=
    C_\frac{3}{2} \Gamma(\tfrac{3}{2}) X^\frac{3}{2}
    +
    \Big(\frac{\pi^2}{6}-1\Big) X
    +
    C_\frac{1}{2} \sqrt \pi X^{\frac{1}{2}} \\
    &\quad +
    \sum_{\pm t_j} C_{\pm t_j} \Gamma(\tfrac{1}{2} \pm it_j) X^{\frac{1}{2} \pm it_j}
    +
    O(X^{\frac{1}{4} + \epsilon}),
  \end{split}
  \end{equation*}
  where $\{t_j\}$ range over the types of the basis of Maass forms.
  The residues $C_{\pm t_j}$ are the coefficients in~\eqref{eq:gross_residues},
  and $C_{\frac{1}{2}}$ is the residue of $4 \mathcal{R}_1^{+}(s-1, 0)/(s-1)$ at
  $s = \tfrac{1}{2}$.
\end{theorem}

Although this is more explicit, the contribution from each
pole $\frac{1}{2} \pm it_j$ oscillates wildly. Further, it follows from Theorem
1 of~\cite{lowryduda2019omega} that the contribution from the spectral poles is
$\Omega_{\pm}(X^{1/2})$. Thus there is a natural square root barrier preventing
additional understandable main terms in the asymptotic.

\begin{remark}
  It should be possible to prove a similar result without assuming the Riemann
  Hypothesis, but with an error term of the shape $X^{1/2} (\log X)^{-A}$. To do
  so, one would play the classical game of choosing contours just within the
  zero-free region of the zeta function up to a certain height.
\end{remark}

As in~\cite[\S8]{HulseGaussSphere}, it is possible to use
Theorem~\ref{thm:DsP2xP2_properties} and Theorem~\ref{thm:smooth_cutoff} to
produce an asymptotic for the continuous Laplace transform.

\begin{theorem}\label{thm:continuous_laplace}
  The Laplace transform of the second moment of the lattice point discrepancy satisfies
  \begin{equation*}
    \int_0^\infty P_2(t)^2 e^{-t/X} \, dt = C_\frac{3}{2} \Gamma(\tfrac{3}{2}) X^\frac{3}{2} - X + O_\epsilon\left(X^{\frac{1}{2}+\epsilon}\right),
  \end{equation*}
  for any $\epsilon > 0$, where $C_\frac{3}{2}$ is defined as in Theorem~\ref{thm:DsP2xP2_properties}.
\end{theorem}

To prove this, we use the identity
\begin{equation} \label{eq:laplace_decomposition}
  P_2(t)^2 = P_2(\lfloor t \rfloor)^2 + \pi^2 (\lfloor t \rfloor - t)^2 + 2\pi P_2(\lfloor t \rfloor)(\lfloor t \rfloor -t)
\end{equation}
and consider the Laplace transform of each term in~\eqref{eq:laplace_decomposition} in turn.

\begin{lemma}[First term in the Laplace transform of~\eqref{eq:laplace_decomposition}]\label{lem:laplace_decomp_1}
  We have
  \begin{equation*}
    \int_0^\infty P_2(\lfloor t \rfloor)^2 e^{-t/X} \, dt = C_\frac{3}{2} \Gamma(\tfrac{3}{2}) X^\frac{3}{2} + \Big(\frac{\pi^2}{6}-1\Big) X + O\left(X^{\frac{1}{2}+\epsilon}\right).
  \end{equation*}
\end{lemma}

\begin{proof}
  This follows from the direct computation
  \begin{align*}
    \int_0^\infty P_2(\lfloor t \rfloor)^2 e^{-t/X} \, dt
    &= X(1-e^{-1/X}) \sum_{n \geq 0} P_2(n)^2 e^{-n/X}
  \end{align*}
  and Theorem~\ref{thm:smooth_cutoff}.
\end{proof}

An estimate for the Laplace transform of the second term in~\eqref{eq:laplace_decomposition} follows by specializing~\cite[Lemma~8.4]{HulseGaussSphere} to the case $k=2$.  However, a far simpler proof is available when the dimension is two.

\begin{lemma}[Second term in the Laplace transform of~\eqref{eq:laplace_decomposition}]\label{lem:laplace_decomp_2}
  We have
  \begin{equation*}
    \pi^2 \int_0^\infty (\lfloor t \rfloor - t)^2 e^{-t/X} \, dt = \frac{\pi^2}{3} X + O(1).
  \end{equation*}
\end{lemma}

\begin{proof}
  We compute
  \begin{align}
    &\pi^2 \int_0^\infty (\lfloor t \rfloor - t)^2 e^{-t/X} \, dt
    =
    \pi^2 \sum_{n \geq 0} \int_n^{n+1} (n^2 -2nt+t^2)e^{-t/X} \, dt \notag
    \\
    &\quad = \pi^2 X e^{-1/X}\left(2e^{1/X} X^2 -2X^2 -2X-1\right)\sum_{n \geq 0}  e^{-n/X}.
    \label{line:lemma_asymptotic}
  \end{align}
  Summing the geometric series, computing a series expansion at infinity, and simplifying completes the proof.
\end{proof}

\begin{remark}
  We note that it is possible to write a more complete asymptotic by using more
  terms in the expansion of~\eqref{line:lemma_asymptotic}.
  But keeping the first term and $O(1)$ error is sufficient to ensure that
  Lemma~\ref{lem:laplace_decomp_2} is not the main source of error in
  Theorem~\ref{thm:continuous_laplace}.
\end{remark}

The analysis of the third term relies on the meromorphic properties of the Dirichlet series with coefficients $P_2(n)$.  Again, proofs in general dimension $k \geq 3$ greatly simplify in dimension $2$.

\begin{lemma}[Third term in the Laplace transform of~\eqref{eq:laplace_decomposition}]\label{lem:laplace_decomp_3}
  We have
  \begin{equation*}
    2\pi \int_0^\infty P_2(\lfloor t \rfloor)(\lfloor t \rfloor -t)e^{-t/X} \, dt = -\frac{\pi^2}{2} X + O(1).
  \end{equation*}
\end{lemma}

\begin{proof}
  Splitting the bounds of integration at integers and summing gives
  \begin{align*}
    &I_3
    :=
    2\pi \int_0^\infty P_2(\lfloor t \rfloor)(\lfloor t \rfloor -t)e^{-t/X} \, dt
    =
    2\pi \sum_{n \geq 0} P_2(n) \int_n^{n+1} (n-t)e^{-t/X} \, dt \\
    &\quad= -2\pi X\big(X-Xe^{-1/X} - e^{-1/X}\big) \sum_{n \geq 0} P_2(n) e^{-n/X},
  \end{align*}
The series that remains may be written as
\begin{align*}
&\sum_{n \geq 0}\sum_{m \leq n} r_2(m) e^{-\frac{n}{X}}-\pi \sum_{n \geq 0} n e^{-\frac{n}{X}}=\sum_{m,h \geq 0} r_2(m) e^{-\frac{m+h}{X}} - \frac{\pi e^{-1/X}}{(1-e^{-1/X})^2}\\
&\qquad =(1-e^{-1/X})^{-1} \sum_{m \geq 0} r_2(m) e^{-\frac{m}{X}} -\frac{\pi e^{-1/X}}{(1-e^{-1/X})^2}. 
\end{align*}
The $m$-sum in the line above is $\theta^2(i/2\pi X)$, which we write as $\pi X \theta^2(\pi i X/2)$ by the functional equation for $\theta(z)$.  Series expansions and the estimate $\theta^2(\pi i X/2) = \sum_{m \geq 0} r_2(m) e^{-\pi^2 m X} = 1+ O(e^{-\pi^2 X})$ give the result.
\end{proof}

Combining Lemmas~\ref{lem:laplace_decomp_1},~\ref{lem:laplace_decomp_2}, and~\ref{lem:laplace_decomp_3} gives a proof of Theorem~\ref{thm:continuous_laplace}.

\section{Estimates for Correlation Sums}\label{sec:D2}

Recall the definition
\begin{equation*}
  D_2(s;h) := \sum_{n \geq 0} \frac{r_2(n+h)r_2(n)}{(n+h)^s}.
\end{equation*}
In section~\S\ref{ssec:D2_Z2_mero}, we saw that $D_2(s;h)$ has a meromorphic
continuation to  $\mathbb{C}$ given by~\eqref{eq:D_spectral}. Further, $D_2$ has
a pole at $s = 1$ with residue given by~\eqref{eq:D_residue_1}, and is otherwise
analytic for $\Re s > \frac{1}{2}$.

In this section, we use this information to produce smooth and sharp estimates
for the shifted convolution sums on $r_2(n)r_2(n+h)$ and prove the following
theorem.

\begin{theorem}\label{thm:chamizo}
  Write $h = 2^\alpha h'$ where $2 \nmid h'$.
  For any $\epsilon > 0$, we have
  \begin{equation}\label{eq:chamizo_smoothed}
    \sum_{n \geq 1} r_2(n) r_2(n+h) e^{-n/X}
    =
    C_h X
    + O_\epsilon \big( X^{\frac{1}{2}
    + \epsilon} e^{h/X} h^\Theta \big),
  \end{equation}
  where $\Theta \leq \frac{7}{64}$ denotes the best-known progress towards the
  (non-archimedean) Ramanujan conjecture, the implicit constant in the error
  term is independent of $h$, and where
  \begin{equation} \label{eq:Ch_smorgasbord}
    C_h
    =
    4\pi^2\big(\varphi_{\infty h}(1)
    + \varphi_{0h}(1)\big)
    =
    \big \lvert 2^{\alpha + 1}
    - 3 \big \rvert \frac{8\sigma_1(h')}{h}
    =
    \frac{8(-1)^h}{h} \!\sum_{d \mid h} (-1)^d d.
  \end{equation}
  The error term in~\eqref{eq:chamizo_smoothed} is
  $O(X^{\frac{1}{2}+\epsilon})$, uniformly for $h \ll X$.

  Correspondingly, we have the weak sharp estimate
  \begin{equation}\label{eq:chamizo_compare}
    \sum_{n \leq X} r_2(n+h) r_2(n) = C_h X + O_\lambda\big( (X+h)^{1 - \lambda} + h \big)
  \end{equation}
  for some $\lambda > 0$, for the same constant $C_h$ as above.
\end{theorem}

\begin{remark}
Theorem~\ref{thm:chamizo} recovers the leading term evaluation of the
correlation sums investigated by Chamizo~\cite{Chamizo1999} and
Ivi\'c~\cite{Ivic1996}, as well as some power savings. The sharp
sum~\eqref{eq:chamizo_compare} is the same sum appearing
in~\eqref{eq:convolution_ivic}. Preliminary investigation suggests that one
could explicitly take $\lambda = 1/8$ using trivial bounds on $D_2(s; h)$. A
deeper and more sophisticated analysis would produce better $\lambda$, but it is
not clear whether the approach outlined here would allow one to improve upon the
bound $\lambda = \frac{1}{3} - \epsilon$ achieved by Chamizo.
\end{remark}

Analyzing the meromorphic continuation of $D_s(s;h)$ given in~\eqref{eq:D_spectral}, one can see that $D_2(s;h)$ is of polynomial growth in vertical strips.
It is therefore straightforward to estimate the integral
\begin{equation}\label{eq:rr_I}
  \frac{1}{2\pi i} \int_{(4)} D_2(s; h) X^s \Gamma(s) ds = \sum_{n \geq 1} r_2(n+h) r_2(n) e^{-(n+h)/X}
\end{equation}
by shifting contours.
Our analysis follows the decomposition of $D_2(s;h)$ given in~\eqref{eq:D_spectral}.
The non-spectral part of $D_2(s;h)$ can be evaluated explicitly through the Mellin inversion identity
\begin{equation*}
  \frac{1}{2\pi i} \int_{(4)} \frac{4\pi^2}{s - 1} \frac{\varphi_{\infty h}(1) + \varphi_{0 h}(1)}{h^{s-1}} X^s \Gamma(s) ds
  =
  4\pi^2(\varphi_{\infty h}(1) + \varphi_{0 h}(1)) Xe^{-h/X}.
\end{equation*}

To bound the contribution from the discrete and continuous components, we shift the line of $s$-integration to $(\frac{1}{2} + \epsilon)$ and bound the integrals.
The $h$-dependence within the discrete spectrum is determined by $\rho_j(h)/h^{s - \frac{1}{2}}$.
Noting that $\rho_j(h) / h^{s - \frac{1}{2}} \ll_{j} h^{\Theta -\epsilon/2}$, where $\Theta \leq \frac{7}{64}$ denotes the best progress towards the (non-archimedean) Ramanujan conjecture, it follows that the discrete contribution is $O_\epsilon(X^{\frac{1}{2} + \epsilon} h^\Theta)$.

Similarly, the $h$-dependence within the continuous spectrum is determined by ${\varphi_{\mathfrak{a} h}(\frac{1}{2} - z)}/ h^{s - \frac{1}{2} - z}$.
From the estimate $\zeta(1+ i t)^{-1} \ll \log(1+ \vert t \vert)$, we obtain $\varphi_{\mathfrak{a} h}(\frac{1}{2} + it) \ll d(h) \log (1+\vert t \vert)$. Thus the continuous contribution is
\begin{equation*}
  O_\epsilon \Big( \frac{d(h)}{h^{\epsilon}} X^{\frac{1}{2}+\epsilon} \Big)
  =
  O_\epsilon(X^{\frac{1}{2}+\epsilon}).
\end{equation*}
We conclude that
\begin{equation*}
  \sum_{n \geq 1} r_2(n+h)r_2(n) e^{-(n+h)/X}
  =
  C_h X e^{-h/X} + O_\epsilon \big( X^{\frac{1}{2}+\epsilon} h^\Theta \big),
\end{equation*}
in which $C_h :=4\pi^2(\varphi_{\infty h}(1)+\varphi_{0h}(1))$.
Multiplying by $e^{h/X}$ proves~\eqref{eq:chamizo_smoothed}.

Using that $D_2(s; h)$ has polynomial growth in vertical strips and nonnegative coefficients, one can combine smoothed integral estimates of the form
\begin{equation}\label{eq:generic_integral_cutoff}
  \frac{1}{2\pi i} \int_{(4)} D_2(s; h) X^s V(s) ds = \sum_{n \geq 1} r_2(n+h) r_2(n) v(\tfrac{n+h}{X}),
\end{equation}
where $V$ and $v$ are Mellin transform pairs, in order to produce estimates for the sharp sums.
In particular, techniques in~\cite[\S5.5.2 and \S5.5.3]{LowryDudaThesis} or~\cite[\S7]{HulseGaussSphere} show how to combine smoothed integrals of the form~\eqref{eq:generic_integral_cutoff} to assemble the weak sharp sum estimate~\eqref{eq:chamizo_compare}.
Similarly, one can prove weak short-interval estimates.

It remains to see that $C_h$ may be written using the alternate expressions given in~\eqref{eq:Ch_smorgasbord}. The equivalence of the two expressions for $C_h$ on the right of~\eqref{eq:Ch_smorgasbord} is straightforward using multiplicativity. We now note that $C_h$ is of the form presented in~\cite[Theorem~3.2]{Chamizo1999}.

\begin{lemma}
  Suppose $h = 2^\alpha h'$ where $2 \nmid h'$.
  Then
  \begin{equation}\label{eq:Chamizo_ID_I}
C_h = 4\pi^2 \big( \varphi_{\infty h} (1) + \varphi_{0 h}(1) \big)
    =
     \frac{8\sigma_1(h')}{h} \big \lvert 2^{\alpha + 1} - 3 \big \rvert.
  \end{equation}
\end{lemma}

\begin{proof}
By~\eqref{eq:Eisen_coeffs}, $C_h$ can be written as
\begin{equation*}
  4\pi^2
  \bigg(
    \frac{\sigma^{(2)}_{-1}(h)}{4 \zeta^{(2)}(2)}
    +
    \frac{\frac{1}{4}\sigma_{-1}(\frac{h}{4}) - \frac{1}{8}\sigma_{-1}(\frac{h}{2})}{\zeta^{(2)}(2)}
  \bigg)
  =
   \frac{8\sigma_{1}(h')}{h} \Big( 2^\alpha + 4\sigma_{1}(2^{\alpha-2}) - \sigma_{1}(2^{\alpha-1})  \Big),
\end{equation*}
where
$\sigma_1(x) = 0$ if $x$ is not a positive integer. A simple case analysis of the last parenthetical expression above yields the proof.
\end{proof}

\bibliographystyle{alpha}
\bibliography{compiled_bibliography}

\end{document}